{\obeyspaces\global\let =\ }

\font\ttten=cmtt10

\newdimen\outputBaseLineSkip
\newskip\beginOutputSkip
\newskip\endOutputSkip
\def\looserOutput#1{%
  \advance\beginOutputSkip by #1
  \advance\endOutputSkip by #1
}
\def\tighterOutput#1{%
  \advance\beginOutputSkip by -#1
  \advance\endOutputSkip by -#1
}

\def\beginOutput{%
    \par
    \penalty -150
    \penalty -150
    \begingroup
      \def\\{%
          \leavevmode
          \hss
          \endgraf
          \penalty 150
          }
      \ttten
      \parindent = 24pt
      \def\${\char`\$}
      \def\{{\char`\{}
      \def\}{\char`\}}
      \catcode`\_=\the\catcode`a
      \catcode`\^=\the\catcode`a
      \catcode`\#=\the\catcode`a
      \catcode`\~=\the\catcode`a
      \catcode`\&=\the\catcode`a
      \parskip=0pt
      \lineskip=0pt
      \obeyspaces
      }
\def\emptyLine{%
    \penalty -100
    \penalty -100
    }
\def\endOutput{%
    \endgroup
    \par
    \penalty -150
    \penalty -150
    \noindent}

\documentclass[a4paper,reqno]{amsart}

\makeatletter
\@namedef{subjclassname@2020}{%
  \textup{2020} Mathematics Subject Classification}
\makeatother

\usepackage[toc]{appendix}
\usepackage{amsmath}
\usepackage{amssymb}
\usepackage{amsthm}
\usepackage{mathrsfs}
\usepackage{latexsym}
\usepackage{amscd}
\usepackage{xypic}
\xyoption{curve}
\usepackage{ifthen} 
\usepackage{hyperref} 
\usepackage{graphicx}
\usepackage{enumerate}
\usepackage{enumitem}
\usepackage{rotating}
\usepackage{xcolor}
\usepackage{mathtools}
\usepackage{todonotes}
\usepackage{color,soul}

\usepackage{float}
\restylefloat{table}

\numberwithin{equation}{section}

\def\cocoa{{\hbox{\rm C\kern-.13em o\kern-.07em C\kern-.13em o\kern-.15em A}}}

\newtheorem{theorem}{Theorem}[section]
\newtheorem{question}[theorem]{Question}

\newtheorem{proposition}[theorem]{Proposition}
\newtheorem{corollary}[theorem]{Corollary}

\theoremstyle{definition}
\newtheorem{remark}[theorem]{Remark}
\newtheorem{definition}[theorem]{Definition}
\newtheorem{example}[theorem]{Example}

\newcommand {\PGL}{\mathrm{PGL}}
\newcommand {\pf}{\mathrm{pf}}

\newcommand {\coker}{\mathrm{coker}}

\newcommand {\depth}{\mathrm{depth}}
\newcommand {\reg}{\mathrm{reg}}
\newcommand {\sHom}{\mathcal{H}\kern -0.25ex{\mathit om}}
\newcommand {\sExt}{\mathcal{E}\kern -0.25ex{\mathit xt}}
\newcommand {\sTor}{\mathcal{T}\kern -0.25ex{\mathit or}}
\newcommand {\im}{\mathrm{im}}
\newcommand {\rk}{\mathrm{rk}}
\newcommand {\pd}{\mathrm{pd}}

\newcommand {\Ext}{\mathrm{Ext}}
\newcommand {\Hom}{\mathrm{Hom}}
\newcommand {\Aut}{\mathrm{Aut}}
\newcommand {\Hilb}{\mathcal{H}\kern -0.25ex{\mathit ilb\/}}

\newcommand {\quantum}{k}

\newcommand {\field}{\mathbf k}

\newcommand {\cA}{\mathcal{A}}
\newcommand {\cB}{\mathcal{B}}

\newcommand {\cU}{\mathcal{U}}
\newcommand {\cV}{\mathcal{V}}

\newcommand{\cC}{{\mathcal C}}

\newcommand{\cE}{{\mathcal E}}
\newcommand{\cF}{{\mathcal F}}
\newcommand{\cM}{{\mathcal M}}
\newcommand{\cN}{{\mathcal N}}
\newcommand{\cO}{{\mathcal O}}
\newcommand{\cG}{{\mathcal G}}

\newcommand{\cI}{{\mathcal I}}

\newcommand {\bZ}{\mathbb{Z}}

\newcommand {\bC}{\mathbb{C}}
\newcommand {\bP}{\mathbb{P}}

\newcommand {\fK}{\frak{K}}
\newcommand {\fH}{\frak{H}}
\newcommand {\fX}{\frak{X}}

\newcommand{\Pic}{\operatorname{Pic}}

\newcommand{\CH}{\operatorname{CH}}

\def\p#1{{\bP^{#1}}}

\def\ga#1{{{\accent"12 #1}}}
 
\def\mapright#1{\mathbin{\smash{\mathop{\longrightarrow}
\limits^{#1}}}}

\title[Steiner representations of hypersurfaces]{Steiner representations of hypersurfaces}

\thanks{The  authors are members of GNSAGA group of INdAM}



\author[V. Antonelli, G. Casnati]{Vincenzo Antonelli, Gianfranco Casnati}

\begin{document}

\maketitle

\begin{abstract}
Let $X\subseteq\p{n+1}$ be an integral hypersurface of degree $d$. We show that each locally Cohen--Macaulay instanton sheaf $\cE$ on $X$ with respect to $\cO_X\otimes\cO_{\p{n+1}}(1)$ in the sense of \cite[Definition 1.3]{An--Ca1} yields the existence of Steiner bundles $\cG$ and $\cF$ on $\p{n+1}$ of the same rank $r$ and a morphism $\varphi\colon \cG(-1)\to\cF^\vee$ such that the form defining $X$ to the power $\rk(\cE)$ is exactly $\det(\varphi)$. We inspect several examples for low values of $d$, $n$ and $\rk(\cE)$. In particular, we show that the form defining a smooth integral surface in $\p3$ is the pfaffian of some skew--symmetric morphism $\varphi\colon \cF(-1)\to\cF^\vee$, where $\cF$ is a suitable Steiner bundle on $\p3$ of sufficiently large even rank.
\end{abstract}

\section{Introduction}
The description of hypersurfaces in $\p{N}$ as  zero loci of suitable square matrices (possibly with some further properties, e.g. with linear entries, symmetric, skew--symmetric, etc.) is a very classical topic in algebraic geometry. We refer the interested reader to the historic comments in \cite{Bea} for a short list of references. 

Such a problem can been placed in the more general framework of the study of {\em Ulrich sheaves} on an irreducible projective scheme $P$ of dimension $N\ge1$ with respect to a very ample line bundle $\cO_P(H)$, i.e. non--zero sheaves $\cF$ on $P$ satisfying
$$
h^i\big(\cF(-(i+1)H)\big)=h^j\big(\cF(-jH)\big)=0
$$
for $i\le N-1$ and $j\ge1$: see \cite{Bea,E--S--W} for further details. While the problem of the existence of Ulrich sheaves on projective varieties is wide open in general, the Horrocks' theorem completely characterizes the ones on $\p N$: indeed they are only the direct sums of copies of  $\cO_{\p{N}}$.

In \cite[Proposition 1.11]{Bea} the author shows that on a hypersurface $X\subseteq\p{n+1}$ of degree $d$ endowed with $\cO_X(h):=\cO_X\otimes\cO_{\p{n+1}}(1)$ each $h$--Ulrich sheaf $\cE$ fits into an exact sequence of the form
\begin{equation*}
0\longrightarrow\cO_{\p{n+1}}(-1)^{\oplus r}\longrightarrow\cO_{\p{n+1}}^{\oplus r}\longrightarrow\cE\longrightarrow0,
\end{equation*}
or, in other words, has a resolution in terms of a suitable endomorphism of degree one of an Ulrich bundle on $\p{n+1}$. Moreover, the following assertions hold if $X$ is  smooth  and integral (see \cite[Corollaries 1.12 and  2.4]{Bea}):
\begin{itemize}
\item the form defining $X$ is the determinant of a $d\times d$ matrix with linear entries if and only if $X$ supports a $h$--Ulrich line bundle $\cE$;
\item the form defining $X$ is the pfaffian of a $2d\times 2d$ skew--symmetric matrix with linear entries if and only if $X$ supports a $h$--Ulrich bundle of rank two  $\cE$ such that $\det(\cE)=\cO_X((d-1)h)$.
\end{itemize}

In \cite{An--Ca1}, we somehow relaxed the notion of Ulrich sheaf by introducing the definition of $H$--instanton sheaf. More precisely, a non--zero sheaf $\cF$ on $P$ is {\em $H$--instanton with quantum number $\quantum$} if
\begin{gather*}
h^0\big(\cF(-H)\big)=h^N\big(\cF(-NH)\big)=0,\qquad h^1\big(\cF(-H)\big)=h^{N-1}\big(\cF(-NH)\big)=\quantum,\\
h^i\big(\cF(-(i+1)H)\big)=h^j\big(\cF(-jH)\big)=0
\end{gather*}
for $1\le i\le 
N-2$ and $2\le j\le N-1$ (notice that the equalities in the latter row actually contribute to the definition only when $N\ge3$). In particular, $H$--Ulrich sheaves are exactly $H$--instanton sheaves with minimal quantum number $\quantum=0$. Moreover, each $H$--instanton sheaf is $H$--Ulrich when $n=1$.

In this paper we show that instanton sheaves on  hypersurfaces $X\subseteq \p{n+1}$ have a resolution in terms of certain bundles which have very few non--zero cohomology groups, namely bundles which are linear or Steiner according the following definition (where {\em generically non--zero} means non--zero at the generic point, so that the torsion free part of the sheaf is not zero).

\begin{definition}
\label{dLinearSteiner}
Let $P$ be an irreducible projective scheme of dimension $N\ge1$ endowed with an ample and globally generated line bundle $\cO_P(H)$.

A generically non--zero coherent sheaf  $\cF$ on $P$ is called $H$--linear if the following finite set of properties hold:
\begin{itemize}
\item $h^0\big(\cF(-H)\big)=h^N\big(\cF(-NH)\big)=0$;
\item $h^i\big(\cF(-(i+1)H)\big)=h^{j}\big(\cF(-jH)\big)=0$ if $1\le i\le N-2$ and $2\le j\le N-1$.
\end{itemize}
If $\cF$ is a $H$--linear sheaf on $X$, then it is called $H$--Steiner if the following additional property holds:
\begin{itemize}
\item $h^1\big(\cF(-H)\big)=0$.
\end{itemize}
\end{definition}

Thanks to \cite[Corollary 4.3]{An--Ca1} we know that $H$--instanton sheaves are $H$--linear, but the converse is clearly not true.

\medbreak

In Section \ref{sGeneral} we recall some general facts that will be used in the paper. In Section \ref{sLinear} we deal with linear sheaves, giving a general characterization in Proposition \ref{pCharacterizationLinear} of their whole cohomology and bounding in Proposition \ref{pRegularity} their Castelnuovo--Mumford regularity. We also prove, motivating our terminology, that $\cO_{\p N}(1)$--linear sheaves are exactly the non--zero sheaves which are cohomology of a linear monad in Proposition \ref{pLinearSpace}. 

To this purpose recall that if $\varphi\colon\cA\to\cB$ is a morphism of locally free sheaves on $P$  we can define the degeneracy loci
$$
D_{r}(\varphi):=\{\ x\in P\ \vert\ \rk(\varphi_x)\le r\ \},
$$
which have a natural reduced scheme structure induced by their inclusion in $P$.

The main goal of the section is to show how one can derive linear sheaves on some hypersurfaces $X\subseteq P$ from linear bundles on $P$, when $P$ is smooth of dimension $N$ (an {\em $N$--fold} for short: in this case $\omega_P=\cO_P(K_P)$ denotes its canonical line bundle). To this purpose, recall that the rank of a sheaf on an integral scheme is defined as the rank of its stalk at the generic point. Moreover, if $\cF$ is a sheaf on the $N$--fold $P$, then its {\em Ulrich dual} with respect to an ample and globally generated line bundle $\cO_P(H)$ is $\cF^{U,H}:=\cF^\vee((N+1)H+K_P)$.

\begin{theorem}
\label{tHypersurface}
Let $P$ be an $N$--fold with $N\ge3$ endowed with an ample and globally generated line bundle $\cO_P(H)$.

If $\cG$ and $\cF$ are $H$--linear bundles of the same rank $r$ on $P$, $\varphi\colon \cG(-H)\to\cF^{U,H}$
is injective, $X:=D_{r-1}(\varphi)$ is irreducible and $\cO_X(h):=\cO_X\otimes\cO_P(H)$, then $X$ is integral  with $\dim(X)=N-1$ and $\cE:=\coker(\varphi)$ is a locally Cohen--Macaulay $h$--linear sheaf on $X$ such that 
\begin{gather*}
h^1\big(\cE(-h)\big)=h^{N-1}\big(\cF(-NH)\big),\qquad h^{N-2}\big(\cE(-(N-1)h)\big)=h^{N-1}\big(\cG(-NH)\big),\\
\begin{aligned}
rH^{N}&=\rk(\cE)h^{N-1}+h^1\big(\cE(-h)\big)+h^{N-2}\big(\cE(-(N-1)h)\big)\\
&-h^1\big(\cF(-H)\big)+h^1\big(\cG(-H)\big).
\end{aligned}
\end{gather*}
\end{theorem}

As an immediate by--product we obtain the following method for constructing $h$--instanton sheaves on hypersurfaces.

\begin{corollary}
\label{cHypersurface}
Let $P$ be an $N$--fold with $N\ge3$ endowed with an ample and globally generated line bundle $\cO_P(H)$.

If $\cG$ and $\cF$ are a $H$--linear bundles of the same rank $r$ on $P$, $\varphi\colon \cG(-H)\to\cF^{U,H}$
is injective, $X:=D_{r-1}(\varphi)$ is irreducible, $\cO_X(h):=\cO_X\otimes\cO_P(H)$ and 
$$
h^{N-1}\big(\cF(-NH)\big)=h^{N-1}\big(\cG(-NH)\big)=\quantum,
$$
then $X$ is integral with $\dim(X)=N-1$ and $\cE:=\coker(\varphi)$ is a locally Cohen--Macaulay $h$--instanton sheaf on $X$ with quantum number
$\quantum$ and
$$
rH^N=\rk(\cE)h^{N-1}+2\quantum.
$$
\end{corollary}

In Section \ref{sSteiner} the notion of Steiner bundle is inspected in a similar way. In particular we characterize in Proposition \ref{pCharacterizationSteiner} the whole cohomology of a Steiner sheaf, we describe in Proposition \ref{pSteinerResolution} Steiner sheaves on irreducible subschemes of $\p N$ in terms of their minimal free resolution, showing in Corollary \ref{cSteinerSpace} that our definition matches the classical one (see \cite[Definition 3.1]{Do--Ka}). 

If $P\cong\p {n+1}$ and $\cO_P(H)\cong\cO_{\p {n+1}}(1)$, then $\cF^{U,H}\cong\cF^\vee$ and Theorem \ref{tHypersurface} can be reversed as follows.

\begin{theorem}
\label{tHypersurfaceSpace}
Let $X\subseteq\p{n+1}$ with $n\ge 2$ be an integral hypersurface of degree $d$ and $\cO_X(h):=\cO_{X}\otimes\cO_{\p {n+1}}(1)$.

If $\cE$ is a locally Cohen--Macaulay $h$--linear sheaf on $X$, then there exists an exact sequence
\begin{equation}
\label{seqSteinerInstanton}
0\longrightarrow\cG(-1)\mapright\varphi\cF^\vee\longrightarrow\cE\longrightarrow0
\end{equation}
where $\cG$ and $\cF$ are Steiner bundles of the same rank on $\p{n+1}$ and $\det(\varphi)$ is the $\rk(\cE)^{\textrm{th}}$--power of the form defining $X$.
\end{theorem}

As in the previous case the above theorem leads to the following immediate result on locally Cohen--Macaulay $h$--instanton sheaves on $X$. 

\begin{corollary}
\label{cHypersurfaceSpace}
Let $X\subseteq\p {n+1}$ with $n\ge 2$ be an integral hypersurface of degree $d$ and $\cO_X(h):=\cO_{X}\otimes\cO_{\p {n+1}}(1)$.

If $\cE$ is a locally Cohen--Macaulay $h$--instanton sheaf on $X$ with quantum number $\quantum$, then $\cE$ fits into  sequence \eqref{seqSteinerInstanton} where $\cG$ and $\cF$ are Steiner bundles of the same rank on $\p{n+1}$, $\det(\varphi)$ is the $\rk(\cE)^{\textrm{th}}$--power of the form defining $X$ and
$$
h^{n}\big(\cG(-n-1)\big)=h^{n}\big(\cF(-n-1)\big)=\quantum.
$$
\end{corollary}

Notice that the above statements are well--known for $N=n+1=2$, because in this case each instanton sheaf is actually Ulrich: see \cite[Sections 3, 4 and 5]{Bea}.

The statements above can be interpreted in terms of determinantal representations of the form defining $X$. For this reason we introduce the following definition.

\begin{definition}
Let $X\subseteq\p {n+1}$ with $n\ge 2$ be a hypersurface of degree $d$.

We say that $X$ has a Steiner $r$--representation if there are Steiner bundles $\cG$ and  $\cF$ of the same rank on $\p{n+1}$ such that 
$$
h^{n}\big(\cG(-n-1)\big)=h^{n}\big(\cF(-n-1)\big),
$$
and  a morphism $\varphi\colon\cG(-1)\to\cF^\vee$ such that $\det(\varphi)$ is the $r^{\textrm{th}}$ power of the form defining $X$. In particular, we say that $X$ has a {\em Steiner determinantal representation} if $r=1$ and a {\em Steiner pfaffian representation} if $r=2$. 

If $X$ has a Steiner representation $\varphi\colon\cG(-1)\to\cF^\vee$, the size of the representation is the common value $\rk(\cG)=\rk(\cF)$.
\end{definition}

Corollaries \ref{cHypersurfaceSpace} and \ref{cHypersurface} say us that the integral hypersurface $X\subseteq\p{n+1}$ of degree $d$ supports a locally Cohen--Macaulay $h$--instanton sheaf $\cE$ with quantum number $\quantum$ and rank $r$ if and only if it has a Steiner $r$--representation of size $\rk(\cE)d+2\quantum$, $\quantum$ being the quantum number of $\cE$.

If $X$ is smooth, then an $h$--instanton sheaf with quantum number $\quantum=0$ is actually an Ulrich bundle. In this particular case $\cE$ is $h$--Ulrich and the same is true for $\cF$ and $\cG$, thus both $\cG$ and $\cF$ are trivial. In this case $X$ has a linear representation i.e. the $r^{\textrm{th}}$ power of the form defining $X$ is the determinant of a matrix of linear forms. 

Thanks to \cite[Corollary 4.3]{An--Ca1}, if $\cE$ is an instanton sheaf on the smooth hypersurface $X\subseteq\p{n+1}$, then  $\dim(\cE_x)=n$ at each $x\in X$. If $\cE$ is also locally Cohen--Macaulay, then it is locally free  on $X$ thanks to the Auslander--Buchsbaum formula (see \cite[Theorem 1.3.3]{Br--He}).  We have two different ways for making the size of the representation as small as possible: indeed, we can minimize either $\quantum$ or $\rk(\cE)$. 

The former case  corresponds to the study of the existence of $h$--Ulrich bundles on $X$, hence in this paper we focus our attention in minimizing $\rk(\cE)$. In particular, in the last three sections, we study what can be said for $\rk(\cE)\le 2$ giving some examples.

More precisely, in Section \ref{sDeterminantal} we study the case of a Steiner determinantal representation of $X$. If $n\ge 3$ (resp. $n=2$) we know that $\Pic(X)$ is principal and generated by $\cO_X(h)$ for the smooth (resp. very general smooth) hypersurface $X$, hence there are no $h$--instanton line bundles on such an $X$ thanks to \cite[Proposition 8.2]{An--Ca1}. In particular, the problem of finding a Steiner determinantal representation of $X$ makes sense only if $n=2$ and in this case we are able to prove the following theorem.

\begin{theorem}
\label{tDeterminantal}
Let $X\subseteq\p3$ be a smooth surface of degree $d\ge 2$. Assume that the characteristic of $\field$ is $0$.

Then $X$ has a Steiner determinantal representation of size $s$ if and only if 
$$
\quantum:=\frac{s-d}2
$$
is a non--negative integer and there is a smooth curve $C\subseteq X$ with
\begin{equation}
\label{InvariantsC}
\begin{gathered}
\deg(C)=\frac12{d(d+2\quantum-1)},\\ p_a(C)=(d+2\quantum){{d+\quantum-2}\choose2}-\quantum{{d+\quantum-1}\choose2}-\quantum{{d+\quantum-3}\choose2}-{{d-1}\choose3}
\end{gathered}
\end{equation}
and such that $h^0\big(\cO_X(C-(\quantum+1)h)\big)=h^0\big(\cO_X((d+\quantum-2)h-C)\big)=0$.
\end{theorem}

Moreover, in the same section, we also inspect in more details the case of smooth surfaces of degree $2\le d\le 4$.

In Section \ref{sPfaffian} we focus on integral hypersurfaces $X\subseteq\p {n+1}$ of degree $d$ supporting  an instanton sheaf $\cE$ which is {\em $\varepsilon$--orientable}, i.e. locally Cohen--Macaulay, reflexive and endowed with a bilinear map $\cE\times\cE\to\det(\cE)\cong\cO_X((d-1)h)$ inducing an isomorphism $\eta\colon\cE\to\cE^\vee((d-1)h)$ such that $\eta^\vee(d-1)=\varepsilon\eta$ where $\varepsilon=\pm1$.

In particular we extend \cite[Theorem B]{Bea} as follows. 

\begin{theorem}
\label{tPfaffian}
Let $X\subseteq\p {n+1}$ with $n\ge 2$ be an integral hypersurface of degree $d$ and $\cO_X(h):=\cO_{X}\otimes\cO_{\p {n+1}}(1)$.

If $\cE$ is an $\varepsilon$--orientable $h$--instanton sheaf on $X$ with quantum number $\quantum$, then there exists an exact sequence
\begin{equation}
\label{seqPfaffian}
0\longrightarrow\cF(-1)\mapright\varphi\cF^\vee\longrightarrow\cE\longrightarrow0
\end{equation}
where $\varphi^\vee(-1)=\varepsilon\varphi$ and $\cF$ is a Steiner bundle with respect to $\cO_{\p{n+1}}(1)$ such that 
$$
h^{n}\big(\cF(-n-1)\big)=\quantum,\qquad \rk(\cF)=\rk(\cE)d+2\quantum.
$$
\end{theorem}

As an application we discuss the existence of a Steiner pfaffian representation of $X$, i.e. of a Steiner bundle $\cF$ on $\p{n+1}$ of even rank $r$ and a skew--symmetric morphism $\varphi\colon\cF(-1)\to\cF^\vee$ such that the form defining $X$ is the pfaffian $\pf(\varphi)$. 

In Proposition \ref{pPfaffianSurface} we show that each surface in $\p3$ has actually a Steiner pfaffian representation, extending well--known classical results for linear pfaffian representations (e.g. see \cite[Proposition 7.6]{Bea} and \cite[Corollary 1.2]{C--K--M}). 

Finally, in Section \ref{sExample}, we collect some further examples of hypersurfaces $X\subseteq\p{n+1}$ of degree $d$ with Steiner pfaffian representation for low values of $d$ and $n$.

\subsection{Acknowledgements}
Both the authors would like to thank J. Jelisiejew for some helpful discussions about the content of Examples \ref{eCubic4} and \ref{eCubicHigher}.

\section{Notation and first results}
\label{sGeneral}
Throughout the whole paper we will work over an algebraically closed field $\field$. For simplicity we will assume that the characteristic of $\field$ is not $2$. Further assumptions on the field will be explicitly specified in each statement when needed. 

The projective space of dimension $N$ over $\field$ will be denoted by $\p N$: $\cO_{\p N}(1)$ will denote the hyperplane line bundle. A projective scheme $P$ is a closed subscheme of some projective space over $\field$: $P$ is a {\em variety} if it is also integral. A {\em manifold} $P$ is a smooth variety: we often use the term $N$--fold for underlying that $P$ is a manifold of dimension $N$.  The structure sheaf of a projective scheme $P$ is denoted by $\cO_P$ and its Picard group by $\Pic(P)$. 

Let $P$ be a projective scheme. For each closed subscheme $Y\subseteq P$ the ideal sheaf $\cI_{Y\vert P}$ of $Y$ in $P$ fits into the exact sequence
\begin{equation}
\label{seqStandard}
0\longrightarrow \cI_{Y\vert P}\longrightarrow \cO_P\longrightarrow \cO_Y\longrightarrow 0.
\end{equation}

If $\cA$ is a coherent sheaf on a projective scheme $P$ we set $h^i\big(\cA\big):=\dim H^i\big(\cA\big)$. If $P$ is  endowed with a globally generated line bundle $\cO_P(H)$, then we have an induced morphism $\phi_H\colon P\to\p M$ where $M+1=h^0\big(\cO_P(H)\big)$ and $\phi_H$ is finite if and only is $\cO_P(H)$ is also ample.

For each $i\ge0$  we set $H^i_*\big(\cA\big):=\bigoplus_{t\in\bZ} H^i\big(\cA(tH)\big)$.
If $S$ is the symmetric $\field$--algebra over $H^0\big(\cO_P(H)\big)$, then  $S\cong \field[x_0,\dots,x_M]$  and $H^i_*\big(\cA\big)$ is naturally an $S$--module. The morphism $\phi_H$ induces a morphism $S\to H^0_*\big(\cO_P\big)$ of $\field$--algebras and we denote by $S[P]$ its image.

The non--zero coherent sheaf $\cA$ on $P$ is called {\em $m$--regular (in the sense of Ca\-stel\-nuo\-vo--Mumford)} if $h^i\big(\cA((m-i)H)\big)=0$ for $i\ge1$ and the {\em regularity} $\reg(\cA)$ of $\cA$ is defined as the minimum integer $m$ such that $\cA$ is $m$--regular. 

\begin{proposition}
\label{pNatural}
Let $P$ be a projective scheme of dimension $n\ge1$ endowed with a globally generated line bundle $\cO_P(H)$.

If $\cA$ is a coherent sheaf on $P$ and there is $m\le n-1$ (resp $m\ge 1$) such that $h^i\big(\cA(-iH)\big)=0$ for each $i\le m$ (resp $i\ge m$), then the following assertions hold.
\begin{enumerate}
\item $h^i\big(\cA(-tH)\big)=0$ for each $t\ge i$ and $i\le m$ (resp. $t\le i$ and $i\ge m$).
\item $h^i\big(\cO_Y\otimes\cA(-iH)\big)=0$ for each $i\le m-1$ (resp. $i\ge m$) and general $Y\in\vert H\vert$.
\end{enumerate}
\end{proposition}
\begin{proof}
See \cite[Proposition 2.1]{An--Ca1}.
\end{proof}

If $\cA$ and $\cB$ are coherent sheaves on an $N$--fold $P$, then the Serre duality holds
\begin{equation}
\label{Serre}
\Ext_P^i\big(\cA,\cB\otimes\omega_P\big)\cong \Ext_P^{N-i}\big(\cB,\cA\big)^\vee
\end{equation}
(see \cite[Proposition 7.4]{Ha3}). 

If $P$ is an $N$--fold, we denote by  $\CH^r(P)$ the group of cycles on $P$ of codimension $r$ modulo rational equivalence: in particular $\CH^1(P)\cong\Pic(P)$ (see \cite[Proposition 1.30]{Ei--Ha2}) and we set $\CH(P):=\bigoplus_{r\ge0}\CH^r(P)$. The Chern classes of a coherent sheaf $\cA$ on $P$ are elements in $\CH(P)$: in particular, when $\cA$ is locally free $c_1(\cA)$ is identified with $\det(\cA)$ via the isomorphism $\CH^1(P)\cong\Pic(P)$.

If $X$ is either a smooth curve or a smooth surface, Hirzebruch--Riemann--Roch formulas for a coherent sheaf $\cA$ are 
\begin{gather}
  \label{RRcurve}
\chi(\cA)=\rk(\cA)\chi(\cO_P)+\deg(c_1(\cA)),\\
\label{RRsurface}
\chi(\cA)=\rk(\cA)\chi(\cO_P)+{\frac12}c_1(\cA)^2-{\frac12}K_Pc_1(\cA)-c_2(\cA),
 \end{gather}
(see \cite[Theorem 14.4]{Ei--Ha2}). 

Let $\cA$ be a rank two vector bundle on an $N$--fold $P$ and let $s\in H^0\big(\cA\big)$. In general its zero--locus
$(s)_0\subseteq P$ is either empty or its codimension is at most
two. Thus, we can always write $(s)_0=Y\cup Z$
where $Z$ has codimension two (or it is empty) and $Y$ has pure codimension
one (or it is empty). 
If $Y=\emptyset$, then $Z$ is locally complete intersection inside $P$, because $\rk(\cA)=2$. In particular, it has no embedded components.

The Serre correspondence reverts the above construction as follows.

\begin{theorem}
  \label{tSerre}
Let $P$ be an $N$--fold with $n\ge2$ and $Z\subseteq P$ a local complete intersection subscheme of codimension two.

If $\det(\cN_{Z\vert P})\cong\cO_Z\otimes\mathcal L$ for some $\mathcal L\in\Pic(P)$ such that $h^2\big(\mathcal L^\vee\big)=0$, then there exists a rank two vector bundle $\cA$ on $P$ satisfying the following properties.
  \begin{enumerate}
  \item $\det(\cA)\cong\mathcal L$.
  \item $\cA$ has a section $s$ such that $Z$ coincides with its zero locus $(s)_0$.
  \end{enumerate}
  Moreover, if $H^1\big({\mathcal L}^\vee\big)= 0$, the above two conditions  determine $\cA$ up to isomorphism.
\end{theorem}
\begin{proof}
See \cite{Ar}.
\end{proof}

In particular the Koszul complex of the section $s\in H^0\big(\cA\big)$
\begin{equation}
  \label{seqSerre}
  0\longrightarrow \cO_P\longrightarrow \cA\longrightarrow \cI_{Z\vert P}\otimes\det(\cA)\longrightarrow 0
\end{equation}
is exact.

\section{Linear sheaves on projective schemes}
\label{sLinear}
In this section we will deal with some properties of linear sheaves on irreducible projective schemes.

In the following proposition we  show that the finite set of vanishing in the cohomology of $\cF$ in the definition give informations on the whole cohomology of $\cF$.

\begin{proposition}
\label{pCharacterizationLinear}
Let $P$ be an irreducible projective scheme of dimension $N\ge1$ endowed with an ample and globally generated line bundle $\cO_P(H)$.

If $\cF$ is a coherent sheaf on $P$, then the following assertions are equivalent.
\begin{enumerate}
\item $\cF$ is a $H$--linear sheaf.
\item $\cF$ is generically non--zero and the following properties hold:
\begin{itemize}
\item $h^0\big(\cF(-(t+1)H)\big)=h^N\big(\cF((t-N)H)\big)=0$ if $t\ge0$;
\item $h^i\big(\cF(-(i+t+1)H)\big)=h^{j}\big(\cF((t-j)H)\big)=0$ if $1\le i\le N-2$, $2\le j\le N-1$ and $t\ge0$.
\end{itemize}
\item For each  finite morphism $p\colon P\to\p N$ with $\cO_P(H)\cong p^*\cO_{\p N}(1)$, then $p_*\cF$ is a linear sheaf on $\p N$ with respect to $\cO_{\p N}(1)$.
\item There is a finite morphism $p\colon P\to\p N$ with $\cO_P(H)\cong p^*\cO_{\p N}(1)$ such that $p_*\cF$ is a linear sheaf on $\p N$ with respect to $\cO_{\p N}(1)$.
\end{enumerate}
\end{proposition}
\begin{proof}
Taking into account of Propositions \ref{pNatural}, \ref{pLinearSpace}, the proof is analogous to the one of the characterization of ordinary instanton sheaves in \cite[Theorem 1.4]{An--Ca1} because the image of the generic point of $P$  via the dominant morphism $p$ is the generic point of $\p n$.
\end{proof}

We show that linear sheaves on the projective spaces are exactly the sheaves that can be obtained as cohomologies of linear monads. Recall that the rank of a sheaf on an integral scheme is the rank of its stalk at the generic point.

\begin{proposition}
\label{pLinearSpace}
Let $N\ge1$.

A coherent sheaf $\cF$ on $\p N$ is linear with respect to $\cO_{\p N}(1)$ if and only if it is both generically non--zero and the cohomology of a monad $\cM^\bullet$ of the form
\begin{equation}
\label{Monad}
0\longrightarrow\cO_{\p N}(-1)^{\oplus m'}\longrightarrow\cO_{\p N}^{\oplus m}\longrightarrow\cO_{\p N}(1)^{\oplus m''}\longrightarrow0.
\end{equation}

Moreover, if this occurs
\begin{equation}
\begin{gathered}
\label{BettiMonad}
m'=h^{N-1}\big(\cF(-N)\big),\qquad m''=h^{1}\big(\cF(-1)\big),\\
m=\chi(\cF)+(N+1)h^{1}\big(\cF(-1)\big)=\rk(\cF)+h^{1}\big(\cF(-1)\big)+h^{N-1}\big(\cF(-N)\big).
\end{gathered}
\end{equation}
\end{proposition}
\begin{proof}
Let $N=1$. In this case $\cF$  is linear with respect to $\cO_{\p1}(1)$ if and only if
$$
h^0\big(\cF(-1)\big)=h^1\big(\cF(-1)\big)=0.
$$
In particular it is necessarily torsion--free, hence $\cF\cong\cO_{\p1}^{\oplus\rk(\cF)}$ and the statement follows trivially. 

Let $N\ge2$. Assume that $\cF$ is the cohomology of monad \eqref{Monad}.  Splitting it into the short exact sequences
\begin{equation}
\label{Display}
\begin{gathered}
0\longrightarrow\cU\longrightarrow\cO_{\p N}^{\oplus m}\longrightarrow\cO_{\p N}(1)^{\oplus m''}\longrightarrow0,\\
0\longrightarrow\cO_{\p N}(-1)^{\oplus m'}\longrightarrow\cU\longrightarrow\cF\longrightarrow0,
\end{gathered}
\end{equation}
as in the proof of  \cite[Proposition 2]{Ja} one easily deduces that $\cF$ is a linear sheaf with respect to $\cO_{\p N}(1)$ and obtains equalities \eqref{BettiMonad}.

Conversely, let $\cF$ be a linear sheaf with respect to $\cO_{\p N}(1)$. If we take a hyperplane $\Sigma\subseteq\p N$ not containing any associated point of $\cF$ we can write the exact sequence
$$
0\longrightarrow\cF(-1)\longrightarrow\cF\longrightarrow\cO_\Sigma\otimes\cF\longrightarrow 0.
$$
From this point on the argument coincides with the one in the proof of  \cite[Theorem 3]{Ja}.
\end{proof}

\begin{corollary}
\label{cCharacteristicLinear}
Let $P$ be an irreducible projective scheme of dimension $N\ge1$ endowed with an ample and globally generated line bundle $\cO_P(H)$.

If $\cF$ is $H$--linear on $P$, then 
\begin{align*}
\chi(\cF(tH))&=(\chi(\cF)+(N+1)h^1\big(\cF(-H)\big)){{N+t}\choose N}\\
&-h^1\big(\cF(-H)\big){{N+t+1}\choose N}-h^{N-1}\big(\cF(-NH)\big){{N+t-1}\choose N}.
\end{align*}
\end{corollary}
\begin{proof}
If $p\colon P\to\p N$ is any finite projection such that $\cO_P(H)\cong p^*\cO_{\p N}(1)$, then $\chi(\cF(tH))=\chi((p_*\cF)(t))$ thanks to \cite[Corollary III.11.2 and Exercises III.8.1, III.8.3]{Ha2} and $p_*\cF$ is linear on $\p n$. 

The first equality in the statement follows by computing $\chi((p_*\cF)(t))$ from the exact sequences \eqref{Display}, taking into account the equalities \eqref{BettiMonad}.
\end{proof}

\begin{remark}
\label{rDual}
Let $\cF$ is a locally free sheaf on an $N$--fold $P$. It is an immediate consequence of the definition and of equality \eqref{Serre} that $\cF$ is $H$--linear on $P$ if and only if the same is true for its Ulrich dual $\cF^{U,H}:=\cF^\vee((N+1)H+K_P)$.

In particular, if $P\cong\p N$ and $\cO_P(H)\cong\cO_{\p N}(1)$, then a locally free sheaf $\cF$ is the cohomology of monad \eqref{Monad} if and only if $\cF^\vee$ is the cohomology of the dual monad. This concludes the remark.
\end{remark}

Trivially each extension of $H$--linear sheaves on $P$ is again a $H$--linear sheaf: in particular direct sums of $H$--linear sheaves are still $H$--linear. Below we give some more interesting examples.

\begin{example}
\label{eEasy}
Every $H$--instanton and $H$--Ulrich sheaf on $P$ is automatically a $H$--linear sheaf. E.g. $\cO_{\p N}$ is linear with respect to $\cO_{\p N}(1)$.

Consider monad \eqref{Monad} on $\p N$. If $m'\ne m''$ and either $m\ge m'+m''+N$ or $m\ge\max\{\ m'+m'', 2m'+N-1\ \}$, \cite[Main Theorem]{Flo} implies that the cohomology of the monad is a $h$--linear sheaf on $\p N$ which is not $h$--instanton. E.g. $\Omega_{\p N}^1(1)$ is linear (but not an instanton) with respect to $\cO_{\p N}(1)$.
\end{example}

In the following proposition we bound the regularity of linear sheaves from above. 

\begin{proposition}
\label{pRegularity}
Let $P$ be an irreducible projective scheme of dimension $N\ge1$ endowed with a very ample line bundle $\cO_P(H)$.

If $\cF$ is a $H$--linear sheaf on $P$, then $\reg(\cF)\le h^1\big(\cF(-H)\big)$.
\end{proposition}
\begin{proof} 
Let $p\colon P\to\p N$ be a finite morphism such that $\cO_P(H)\cong p^*\cO_{\p N}(1)$. We know that $p_*\cF$ is the cohomology of a monad $\cM^\bullet$ as in Proposition \ref{pLinearSpace}. Thanks to the equality $h^i\big(\cF((m''-i)H)\big)=h^i\big((p_*\cF)(m''-i)\big)$, it suffices to check that $\reg(p_*\cF)\le h^1\big((p_*\cF)(-1)\big)=m''$, thus we will assume $P\cong\p N$ from now on, whence $\cO_P(H)\cong\cO_{\p N}(1)$.

An analogous statement for vector bundles is proved in \cite[Theorem 3.2]{C--MR1}: nevertheless, the proof therein can be repeated verbatim for each linear sheaf. It follows that $\reg(\cF)\le m''=h^1\big(\cF(-1)\big)$.
\end{proof}

\begin{remark}
If $\cF$ is locally free, then $\cF^{U,H}$ is the cohomology of the dual of monad \eqref{Monad}, hence $\reg(\cF^{U,H})\le m'$. This concludes the remark.
\end{remark}

We close this section by exploiting the interesting relation between linear bundles on an $N$--fold $P$ and on certain hypersurfaces inside it.

\medbreak
\noindent{\it Proof of Theorem \ref{tHypersurface}.}
By definition, we have the exact sequence
\begin{equation}
\label{seqFE}
0\longrightarrow\cG(-H)\mapright\varphi\cF^{U,H}\longrightarrow\cE\longrightarrow0,
\end{equation}
hence $\cE$ is a sheaf supported on $X$. Since $P$ is reduced, the same is true for $X$, hence it is integral. Moreover, $\det(\varphi)$ vanishes exactly on $X$, hence $n=N-1$.

Notice that $\pd_{\cO_{P,x}}\cE_x\le 1$ for each $x\in X$, hence the Auslander--Buchsbaum formula and \cite[Exercise 1.2.26 b)]{Br--He} yield $\depth_{\cO_{X,x}}\cE_x=\depth_{\cO_{P,x}}\cE_x\ge N-1$. Since $\dim(\cE_x)\le\dim(X)=N-1$, we deduce that $\cE$ is locally Cohen--Macaulay.

The cohomology of sequence \eqref{seqFE} suitably twisted then yields
\begin{gather*}
h^i\big(\cE(-(i+1)h)\big)\le h^i\big(\cF^{U,H}(-(i+1)H)\big)+h^{i+1}\big(\cG(-(i+2)H)\big),\\
h^j\big(\cE(-jh)\big)\le h^j\big(\cF^{U,H}(-jH)\big)+h^{j+1}\big(\cG(-(j+1)H)\big).
\end{gather*}
By definition, also thanks to Remark \ref{rDual}, the summands on the right vanish if $0\le i\le N-2$, $0\le i+1\le N$, $2\le j\le N$, $2\le j+1\le N$. These conditions and the restriction $n\ge3$ imply $0\le i\le N-3$ and $2\le j\le N-1$, hence
$$
h^i\big(\cE(-(i+1)h)\big)=h^j\big(\cE(-jh)\big)=0
$$
in the same range. Finally, the cohomology of sequence \eqref{seqFE} suitably twisted and equality \eqref{Serre} yield the two first equalities in the statement.

Computing $\chi$ from sequence \eqref{seqFE} we obtain
$$
\chi(\cE)=\chi(\cF^{U,H})-\chi(\cG(-H))=(-1)^N\chi(\cF(-(N+1)H)-h^1\big(\cG(-H)\big),
$$
because $\cG$ is $H$--linear. By combining the above equality with Corollary \ref{cCharacteristicLinear} we obtain
\begin{equation}
\label{chiEFG}
\chi(\cE)=\chi(\cF)+(N+1)h^1\big(\cF(-H)\big)-(N+1)h^{N-1}\big(\cF(-NH)\big)-h^1\big(\cG(-H)\big).
\end{equation}
If $p\colon P\to\p N$ is finite with $\cO_P(H)\cong p^*\cO_{\p N}(1)$, then $\rk(p_*\cF)=\rk(\cF)H^N$. Thanks to \cite[Corollary III.11.2 and Exercises III.8.1, III.8.3]{Ha2} and equality \eqref{BettiMonad} we deduce that
\begin{equation}
\label{chiF}
\chi(\cF)=\rk(\cF)H^N-Nh^1\big(\cF(-H)\big)+h^{N-1}\big(\cF(-NH)\big).
\end{equation}
A similar argument yields also
\begin{equation}
\label{chiE}
\chi(\cE)=\rk(\cE)h^{N-1}-(N-1)h^1\big(\cE(-h)\big)+h^{N-2}\big(\cE(-(N-1)h)\big).
\end{equation}
The last equation in the statement follows by combining equalities \eqref{chiEFG}, \eqref{chiF} and \eqref{chiE} with the first two equalities in the statement.
\qed
\medbreak

\section{Steiner sheaves on projective schemes}
\label{sSteiner}
In this section we will deal with some properties of Steiner sheaves on irreducible projective schemes, besides the ones already described in the previous section. 

The proof of the following proposition is analogous to the proof of Proposition \ref{pCharacterizationLinear} and of \cite[Theorem 1.4]{An--Ca1}.

\begin{proposition}
\label{pCharacterizationSteiner}
Let $P$ be an irreducible projective scheme of dimension $N\ge1$ endowed with an ample and globally generated line bundle $\cO_P(H)$.

If $\cF$ is a coherent sheaf on $P$, then the following assertions are equivalent.
\begin{enumerate}
\item $\cF$ is a $H$--Steiner sheaf.
\item $\cF$ is generically non--zero and the following properties hold:
\begin{itemize}
\item $h^0\big(\cF(-(t+1)H)\big)=h^N\big(\cF((t-N)H)\big)=0$ if $t\ge0$;
\item $h^i\big(\cF(-(i+t+1)H)\big)=h^{j}\big(\cF((t-j)H)\big)=0$ if $1\le i\le N-2$, $1\le j\le N-1$ and $t\ge0$.
\end{itemize}
\item For each  finite morphism $p\colon P\to\p N$ with $\cO_P(H)\cong p^*\cO_{\p N}(1)$, then $p_*\cF$ is a Steiner sheaf on $\p N$ with respect to $\cO_{\p N}(1)$.
\item There is a finite morphism $p\colon P\to\p N$ with $\cO_P(H)\cong p^*\cO_{\p N}(1)$ such that $p_*\cF$ is a Steiner sheaf on $\p N$ with respect to $\cO_{\p N}(1)$.
\end{enumerate}
\end{proposition}

If $\cF$ is $H$--Steiner on $P$, then $\chi(\cF)=h^0\big(\cF\big)$ and Corollary \ref{cCharacteristicLinear} yields for each $t\ne0$
\begin{equation}
\label{CharacteristicSteiner}
\chi(\cF(tH))=h^0\big(\cF\big){{N+t}\choose N}-h^{N-1}\big(\cF(-NH)\big)
{{N+t-1}\choose N}.
\end{equation}

Let $X\subseteq\p N$ and set $\cO_X(h):=\cO_X\otimes\cO_{\p N}(1)$. If $X$ has dimension $n=1$, then each $h$--Steiner sheaf $\cF$ on $X$ is $h$--Ulrich, hence \cite[Propositions 1.9 and 2.1]{E--S--W} characterize the minimal free resolution of $\cF$ as $\cO_{\p N}$--sheaf. When $n\ge2$, it is still possible to identify $h$--Steiner sheaves via their minimal free resolutions, as in the case of Ulrich sheaves.

\begin{proposition}
\label{pSteinerResolution}
Let $X\subseteq\p N$ be an irreducible scheme of dimension $n\ge2$ and set $\cO_X(h):=\cO_X\otimes\cO_{\p N}(1)$.

A coherent sheaf $\cF$ on $X$ is $h$--Steiner if and only if it is both generically non--zero and the minimal free resolution of $H_*^0\big(\cF\big)$ as a graded module over $S:=\field[x_0,\dots x_N]$ has the form
\begin{equation}
\label{seqSteinerResolution}
\begin{aligned}
0\longrightarrow S(n-N-1)^{\oplus m_{N-n+1}}&\mapright{\alpha_{N-n+1}} S(n-N)^{\oplus m_{N-n}}\mapright{\alpha_{N-n}}\dots \\
&\mapright{\alpha_2} S(-1)^{\oplus m_1}\mapright{\alpha_1} S^{\oplus m_0}\longrightarrow H_*^0\big(\cF\big)\longrightarrow0.
\end{aligned}
\end{equation}
Moreover, if this occurs, then 
\begin{equation}
\label{BettiSteiner}
m_t=h^0\big(\cF\big){{N-n}\choose {t}}+ h^{n-1}\big(\cF(-nh)\big) {{N-n} \choose {t-1}}
\end{equation}
for $0\le t\le N-n+1$.
\end{proposition}
\begin{proof}
If $\cF$ is a $h$--Steiner sheaf, then $h^j\big(\cF(-jh)\big)=0$ for $1\le j\le N$ by definition, hence $\reg(\cF)=0$. In particular the minimal free graded resolution $F_\bullet$ of $H^0_*\big(\cF\big)$ over $S:=\field[x_0,\dots x_N]$ satisfies $F_i\cong S(-i)^{\oplus m_i}$, where $m_i\in\bZ$ thanks to \cite[Corollary 4.17]{Eis}. Moreover \cite[Proposition 77.2 (a)]{Ser} implies $F_i=0$ for $i\ge N+1$.

We certainly have $m_N=0$. Indeed, sheafifying $F_\bullet\to H^0_*\big(\cF\big)$ and splitting it, we obtain the short exact sequence
\begin{equation}
\label{seqA}
0\longrightarrow\cA_{s+1}(-1)\longrightarrow\cO_{\p N}^{\oplus m_s}\longrightarrow\cA_s\longrightarrow0
\end{equation}
for each $0\le s\le N-1$, where we set $\cA_0:=\cF$ and $\cA_{N}:=\cO_{\p N}^{\oplus m_{N}}$. Thus, computing their cohomologies we obtain $0=h^0\big(\cF(-1)\big)=h^N\big(\cA_N(-N-1)\big)=m_N$.

If $n=2$, then $N-n+1=N-1$, hence sequence \eqref{seqSteinerResolution} is exactly the resolution $F_\bullet\to H_*^0\big(\cF\big)$. If $n\ge3$, then $H^0_*\big(\cF\big)$ satisfies the hypothesis of \cite[Proposition 77.2 (b)]{Ser}. Thus the length of  $F_\bullet$ is $N-n+1$ at most and we can conclude as in the case $n=2$.

Conversely, assume that $\cF$ is generically non--zero on $X$ and fits into sequence \eqref{seqSteinerResolution}. As above we obtain the short exact sequence
\eqref{seqA} where $0\le s\le N-n$. Computing their cohomologies we obtain $h^j\big(\cA_s(-t)\big)=h^{j+1}\big(\cA_{s+1}(-t-1)\big)$ for $1\le t\le N$ and $0\le j\le N$. Taking $t:=j$ we obtain
$$
h^j\big(\cF(-jh)\big)=m_{N-n+1}h^{j+N-n+1}\big(\cO_{\p N}(n-N-1-j)\big)=0
$$
in the range $1\le j\le n$. Taking $t:=j+1$ we obtain
$$
h^j\big(\cF(-(j+1)h)\big)=m_{N-n+1}h^{j+N-n+1}\big(\cO_{\p N}(n-N-2-j)\big)=0
$$
in the range $0\le j\le n-2$. Thus $\cF$ is a $h$--Steiner sheaf on $X$.

In the remaining part of the proof we compute the $m_t$'s.  
The cohomology of sequence \eqref{seqA} for $s:=t$ returns
\begin{equation}
\label{Mt}
m_t=\chi(\cA_{t})+\chi(\cA_{t+1}(-1)).
\end{equation}
The cohomologies of sequences \eqref{seqA} for $s:=t+j$ tensored by $\cO_{\p N}(-j)$ yield
$$
\chi(\cA_{t+j}(-j))=-\chi(\cA_{t+1+j}(-1-j))
$$
for each $1\le j\le N$. Thus
\begin{equation}
\label{Vanish}
\begin{aligned}
\chi(\cA_{t+1}(-1))=(-1)^{N-n-t}m_{N-n+1}\chi(\cO_{\p N}(-N+n-1+t))=0,
\end{aligned}
\end{equation}
because $n\ge2$ and $\cA_{N-n+1}=\cO_{\p N}^{\oplus m_{N-n+1}}$. The cohomologies of sequences \eqref{seqA} for $s:=t-i$ tensored by $\cO_{\p N}(i)$ yield
\begin{equation}
\label{chiA}
\begin{aligned}
\sum_{i=1}^t(-1)^{i}m_{t-i}{{N+i}\choose N}&=\sum_{i=1}^t(-1)^{i}(\chi(\cA_{t-i}(i))+\chi(\cA_{t+1-i}(i-1)))\\
&=-\chi(\cA_{t})+(-1)^{t}\chi(\cF(t))
\end{aligned}
\end{equation}

By combining equalities \eqref{Mt}, \eqref{Vanish} and \eqref{chiA} with equality \eqref{CharacteristicSteiner} we obtain the following square system of linear equations in the variables $m_0,\dots m_{N-n+1}$
\begin{equation}
\label{SystemSteiner}
\left\lbrace\begin{array}{l} 
\sum_{i=1}^t(-1)^im_{t-i}{{N+i}\choose N}+m_t=(-1)^th^0\big(\cF\big){{n+t}\choose n}+(-1)^{t+1}\quantum{{n+t-1}\choose n},\\
t=0,\dots,N-n+1,
\end{array}\right.
\end{equation}
where $\quantum:=h^{n-1}\big(\cF(-nh)\big)$. The matrix of the system is lower triangular and its diagonal entries are all equal to $1$. Thus system \eqref{SystemSteiner} has a unique solution exclusively depending on the two parameters $h^0\big(\cF\big)$ and $\quantum$, say
$m_t=h^0\big(\cF\big)u_t+\quantum  v_t$ for $t=0,\dots,N-n+1$ where $u_t$ and $v_t$ are integer--valued functions. 

In particular, the value of $u_t$ can be computed by taking $\quantum=0$. In this case $\cF$ is $h$--Ulrich, hence 
$$
u_t={N-n\choose t}
$$
thanks to \cite[p. 548]{E--S--W}. 

In order to compute $v_t$, we first notice that $v_0=0$ thanks to system  \eqref{SystemSteiner}. Thus it is enough to show that $v_t=u_{t-1}$ for $t\ge1$. We will prove such an equality by induction on $t$. 

When $t=1$  system  \eqref{SystemSteiner} yields $m_1-m_0(N+1)=\quantum-(n+1)h^0\big(\cF\big)$. The equality $m_0=h^0\big(\cF\big)$ yields $v_1=u_0$. Let $t\ge2$ and assume that equality \eqref{BettiSteiner} is true for $s<t$. System \eqref{SystemSteiner} yields
$$
m_s=\sum_{i=1}^s(-1)^{i-1}m_{s-i}{{N+i}\choose N}+(-1)^sh^0\big(\cF\big){{n+s}\choose n}+(-1)^{s+1}\quantum{{n+s-1}\choose n}.
$$
Thus, by inductive hypothesis
\begin{gather*}
\begin{align*}
m_{t-1}=&\sum_{i=1}^{t-1}(-1)^{i-1}\left(h^0\big(\cF\big){{N-n}\choose {t-1-i}}+ \quantum {{N-n} \choose {t-2-1}}\right){{N+i}\choose N}+\\
&+(-1)^{t-1}h^0\big(\cF\big){{n+t-1}\choose n}+(-1)^{t}\quantum{{n+t-2}\choose n},
\end{align*}\\
\begin{align*}
m_{t}=&\sum_{i=1}^{t}(-1)^{i-1}\left(h^0\big(\cF\big){{N-n}\choose {t-i}}+ \quantum {{N-n} \choose {t-i-1}}\right){{N+i}\choose N}+\\
&+(-1)^th^0\big(\cF\big){{n+t}\choose n}+(-1)^{t+1}\quantum{{n+t-1}\choose n}.
\end{align*}
\end{gather*}
Computing $u_{t-1}$ from the first expression and $v_t$ from the second one, we obtain $v_t=u_{t-1}$ because  the summand $i=t$ in the expression of $v_t$ vanishes.

Thus the proof is complete.
\end{proof}

The above proposition returns immediately the following corollary  (see \cite[Theorem 4.2] {Ja--VM}). In particular our definition of Steiner sheaves matches with the classical one from \cite[Section 3]{Do--Ka}.
\begin{corollary}
\label{cSteinerSpace}
Let $N\ge2$.

A coherent sheaf $\cF$ on $\p N$ is Steiner with respect to $\cO_{\p N}(1)$ if and only if it is both generically non--zero and fits into an exact sequence of the form
\begin{equation}
\label{seqSteinerSpace}
0\longrightarrow\cO_{\p N}(-1)^{\oplus m_1}\longrightarrow\cO_{\p N}^{\oplus m_0}\longrightarrow\cF\longrightarrow0.
\end{equation}
Moreover, if this occurs, then
$$
m_0=h^0\big(\cF\big),\qquad m_1=h^{N-1}\big(\cF(-N)\big).
$$
\end{corollary}

\begin{remark}
Let $P$ be integral and endowed with an ample and globally generated line bundle $\cO_P(H)$. 

If $\cF$ is $H$--Steiner and $p\colon P\to\p N$ is a finite morphism induced by $\cO_P(H)$, then $p_*\cF$ is Steiner with respect to $\cO_{\p N}(1)$ and we can apply Corollary  \ref{cSteinerSpace} to it. Since $\deg(p)=\deg(P)$, it follows that $\rk(p_*\cF)=\rk(\cF)\deg(P)$, hence 
$$
h^0\big(\cF\big)=\rk(\cF)\deg(P)+h^{N-1}\big(\cF(-NH)\big)
$$
by confronting the ranks of the sheaves in sequence \eqref{seqSteinerSpace} for $p_*\cF$.

Thus Ulrich sheaves on $P$ are the ones having the maximal number of generators in the class of  aCM sheaves on $P$ of fixed rank (see \cite[Theorem 3.1]{C--H1} and the minimal number of generators among Steiner sheaves on $P$ of fixed rank.
\end{remark}

\begin{example}
\label{eKernel}
Let $a,b,\quantum$ be non--negative integers with $b\le (N+1)\quantum+a$ and consider a sheaf $\cF$ fitting into an exact sequence of the form 
\begin{equation}
\label{seqOmega}
0\longrightarrow\cO_{\p {N}}^{\oplus b}\longrightarrow\cO_{\p {N}}^{\oplus a}\oplus\Omega_{\p {N}}^{N-1}(N)^{\oplus \quantum}\longrightarrow\cF\longrightarrow0.
\end{equation}
If $\cF$ is generically non--zero, then it is easy to check by direct computation that it is a Steiner sheaf with respect to $\cO_{\p N}(1)$ with $h^{N-1}\big(\cF(-N)\big)=\quantum$. In particular, $\cF$ fits into sequence \eqref{seqSteinerSpace} with $m_0=h^0\big(\cF\big)=(N+1)\quantum+a-b$ and $m_1=\quantum$.

Conversely, if $\cF$ is a Steiner sheaf fitting into sequence \eqref{seqSteinerSpace}, then the cohomology of $\cF(-1)$ is as in table 1.
\begin{tiny}
\begin{table}[H]
\label{TableKernel}
\centering
\bgroup
\def\arraystretch{1.5}
\begin{tabular}{ccccc}
\cline{2-4}
\multicolumn{1}{c}{} &\multicolumn{1}{|c|}{$h^N\big(\cF(-N-1)\big)$} & \multicolumn{1}{c|}{0} & \multicolumn{1}{c|}{0} & $q=N$ \\ 
\cline{2-4}
\multicolumn{1}{c}{} &\multicolumn{1}{|c|}{$h^{N-1}\big(\cF(-N-1)\big)$} & \multicolumn{1}{c|}{$h^{N-1}\big(\cF(-N)\big)$} & \multicolumn{1}{c|}0 & $q=N-1$ \\ 
\cline{2-4}
\multicolumn{1}{c}{} &\multicolumn{1}{|c|}{0} & \multicolumn{1}{c|}{0} & \multicolumn{1}{c|}0 & $0\le q\le N-2$ \\ 
\cline{2-4}
&$p=-N-1$& $p=-N$ & $-N+1\le p\le -1$
\end{tabular}
\egroup
\caption{Values of $h^{q}\big(\cF(ph)\big)$ in the range $-N-1\le p\le -1$}
\end{table}
\end{tiny}
\noindent The existence of sequence \eqref{seqOmega} then follows from \cite[Beilinson's theorem (strong form)]{An--Ot1}.
\end{example}

The notion of Steiner sheaf is stable with respect to the restriction on general hyperplane sections.

\begin{proposition}
\label{pHyperplaneSteiner}
Let $P$ be an irreducible projective scheme of dimension $N\ge3$ endowed with an ample and globally generated line bundle $\cO_P(H)$.

If $Y\in \vert H\vert$ is general and $\cF$ is a $H$--Steiner sheaf, then $\cO_Y\otimes\cF$ is a Steiner sheaf with respect to $\cO_Y\otimes\cO_P(H)$ and
$$
h^{N-1}\big(\cF(-NH)\big)=h^{N-2}\big(\cO_Y\otimes\cF(-(N-1)H)\big).
$$
\end{proposition}
\begin{proof}
We can assume that $Y$ does not contain any associated point of $\cF$, hence the natural sequence
$$
0\longrightarrow\cF(-H)\longrightarrow\cF\longrightarrow\cO_Y\otimes\cF\longrightarrow0
$$
is exact. The statement then follows by computing the cohomology of the above sequence after suitable twists.
\end{proof}

When either $P$ or $\cO_P(H)$ satisfy further extra conditions one can say something more. E.g. we can assume that $P$ is an $N$--fold, i.e. it is smooth, hence integral.

\begin{proposition}
\label{pSlope}
Let $P$ be an $N$--fold with $N\ge1$ endowed with an ample and globally generated  line bundle $\cO_P(H)$. Assume that either the characteristic of $\field$ is $0$ or $\cO_P(H)$ is very ample.

If $\cF$ is a $H$--Steiner vector bundle on $P$, then 
\begin{equation}
\label{Slope}
c_1(\cF)H^{N-1}=\frac{\rk(\cF)}2((N+1)H^N+K_PH^{N-1})+h^{N-1}\big(\cF(-NH)\big).
\end{equation}
\end{proposition}
\begin{proof}
Let $N\ge3$. Each general $Y\in\vert H\vert$ is an $(N-1)$--fold by the Bertini theorem (see \cite[Theorem II.8.18 and Corollaries III.7.9, III.10.9]{Ha2}). Let $\cF_Y:=\cO_Y\otimes\cF$, $\cO_Y(H_Y):=\cO_Y\otimes\cO_P(H)$ and assume that the statement holds on such a $Y$. Thus
$$
(N+1)H^N+K_PH^{N-1}=NH^{N-1}_Y+K_YH^{N-2}_Y
$$
thanks to the adjunction formula on $P$ and $c_1(\cF)H^{N-1}=c_1(\cF_Y)H_Y^{N-2}$. It follows that the equality \eqref{Slope} holds for the bundle $\cF$ on $P$ if and only if it holds for the bundle $\cF_Y$ on $Y$, which is a $H_Y$--Steiner bundle thanks to Proposition \ref{pHyperplaneSteiner}. 

Thus it suffices to handle the case $N\le2$, the case $N=1$ being trivial. If $N=2$, then $\chi(\cF(-H))=0$ and $\chi(\cF(-2H))=-h^1\big(\cF(-2H)\big)$. Equality \eqref{RRsurface} yields
\begin{align*}
h^1\big(\cF(-2H)\big)&=\chi(\cF(-H))-\chi(\cF(-2H))\\
&=c_1(\cF)H^{N-1}-\frac{\rk(\cF)}2((N+1)H^N+K_PH^{N-1}).
\end{align*}
Thus the statement is completely proved.
\end{proof}

We are now ready to prove Theorem \ref{tHypersurfaceSpace} stated in the introduction.

\medbreak
\noindent{\it Proof of Theorem \ref{tHypersurfaceSpace}.}
The cohomology of $\cE$ as coherent sheaf on $\p n$ is as in table 3.
\begin{tiny}
\begin{table}[H]
\label{TableEpq}
\centering
\bgroup
\def\arraystretch{1.5}
\begin{tabular}{cccccccc}
\cline{2-7}
\multicolumn{1}{c}{} & \multicolumn{1}{|c|}{0} &\multicolumn{1}{c|}{0} & \multicolumn{1}{c|}{0} & \multicolumn{1}{c|}{0} & \multicolumn{1}{c|}{0} & \multicolumn{1}{c|}{0} & $q={n+1}$ \\ 
\cline{2-7}
\multicolumn{1}{c}{} & \multicolumn{1}{|c|}{$h^{n}\big(\cE(-(n+1)h)\big)$} &\multicolumn{1}{c|}{0} & \multicolumn{1}{c|}{0} & \multicolumn{1}{c|}{0} & \multicolumn{1}{c|}{0} & \multicolumn{1}{c|}{0} & $q=n$ \\ 
\cline{2-7}
\multicolumn{1}{c}{} & \multicolumn{1}{|c|}{$h^{n-1}\big(\cE(-(n+1)h)\big)$} &\multicolumn{1}{c|}{$h^{n-1}\big(\cE(-nh)\big)$} & \multicolumn{1}{c|}{0} & \multicolumn{1}{c|}{0} & \multicolumn{1}{c|}{0} & \multicolumn{1}{c|}{0} & $q=n-1$ \\ 
\cline{2-7}
\multicolumn{1}{c}{} &\multicolumn{1}{|c|}0 & \multicolumn{1}{c|}{0} & \multicolumn{1}{c|}0 & \multicolumn{1}{c|}{0} & \multicolumn{1}{c|}{0} & \multicolumn{1}{c|}{0} & $2\le q\le n-2$ \\ 
\cline{2-7}
\multicolumn{1}{c}{} &\multicolumn{1}{|c|}{0} & \multicolumn{1}{c|}{0} & \multicolumn{1}{c|}{0} & \multicolumn{1}{c|}{0} & \multicolumn{1}{c|}{$h^1\big(\cE(-h)\big)$} & \multicolumn{1}{c|}{$h^1\big(\cE\big)$} & $q=1$ \\ 
\cline{2-7}
\multicolumn{1}{c}{} &\multicolumn{1}{|c|}{0} & \multicolumn{1}{c|}{0} & \multicolumn{1}{c|}{0} & \multicolumn{1}{c|}{0} & \multicolumn{1}{c|}{0} & \multicolumn{1}{c|}{$h^0\big(\cE\big)$} & $q=0$ \\ 
\cline{2-7}
&$p=-n-1$ & $p=-n$ &$1-n\ \le$ &\hskip-1truemm $p\ \ \le\ -2$& $p=-1$ & $p=0$
\end{tabular}
\egroup
\caption{Values of $h^{q}\big(\cE(ph)\big)$ in the range $-n-1\le p\le 0$}
\end{table}
\end{tiny}
\noindent Thanks to \cite[Beilinson's theorem (strong form)]{An--Ot1} we have two sheaves $\cG(-1)\subseteq\widehat{\cF}$ fitting into exact sequences
\begin{equation}
\label{XXX}
\begin{gathered}
0\longrightarrow\widehat{\cF}\longrightarrow\cO_{\p {n+1}}^{\oplus a}\oplus\Omega_{\p {n+1}}^{1}(1)^{\oplus \quantum}\longrightarrow\cO_{\p {n+1}}^{\oplus b}\longrightarrow0,\\
0\longrightarrow\cO_{\p {n+1}}(-1)^{\oplus b'}\longrightarrow\cO_{\p {n+1}}(-1)^{\oplus a'}\oplus\Omega_{\p {n+1}}^{n}(n)^{\oplus \quantum'}\longrightarrow\cG(-1)\longrightarrow0,
\end{gathered}
\end{equation}
where
\begin{gather*}
a:=h^0\big(\cE\big),\qquad b:=h^1\big(\cE\big),\qquad \quantum:=h^1\big(\cE(-h)\big),\\
a':=h^{n}\big(\cE(-(n+1)h)\big),\qquad b':=h^{n-1}\big(\cE(-(n+1)h)\big),\qquad \quantum':=h^{n-1}\big(\cE(-nh)\big)
\end{gather*}
and such that $\cE\cong\widehat{\cF}/\cG(-1)$. Thus we have an exact sequence
$$
0\longrightarrow\cG(-1)\mapright\varphi\widehat{\cF}\mapright q\cE\longrightarrow0
$$
By construction $\widehat{\cF}$ is a vector bundle and $\cG$ is torsion free, thus generically non--zero. Thus both $\cF:=\widehat{\cF}^\vee$ and $\cG$ are Steiner with respect to $\cO_{\p{n+1}}(1)$ and
$$
h^{n}\big(\cF(-n-1)\big)=\quantum,\qquad h^{n}\big(\cG(-n-1)\big)=\quantum'
$$
thanks to Example \ref{eKernel}. In particular we have proved the existence of sequence \eqref{seqSteinerInstanton}. Since the support of $\cE$ is $X$ and $\varphi$ is injective, it follows that $\rk(\cG)=\rk(\cF)$.

If $x\not\in X$, then $\cG_x\cong\cF_x^\vee$. If $x\in X$, then $\dim(\cE_x)=\dim(\cO_{X,x})=n$ because $\cE$ is generically non--zero on $X$ by hypothesis. Since $\cE$ is locally Cohen--Macaulay, it follows that $\depth_{\cO_{\p{n+1},x}}\cE_x=\depth_{\cO_{X,x}}\cE_x=n$ (see  \cite[Exercise 1.2.26 b)]{Br--He}). Thus the Auslander--Buchsbaum formula yields $\pd_{\cO_{\p {n+1},x}}\cE_x=1$, hence $\pd_{\cO_{\p {n+1},x}}\cG_x=0$ (e.g. see \cite[Exercise III.6.5]{Ha2}). We conclude that $\cG$ is locally free on $\p{n+1}$. 

Corollary \ref{cSteinerSpace} implies $\det(\cF)\cong\cO_{\p{n+1}}(\quantum)$ and $\det(\cG)\cong\cO_{\p{n+1}}(\quantum')$, hence 
$$
\det(\varphi)\colon\det(\cG)(-\rk(\cG))\to\det(\cF)^{-1}
$$
represents $\cO_{\p{n+1}}(\rk(\cE)d)$ in $\Pic(\p{n+1})\cong \bZ\cO_{\p{n+1}}(1)$. Since the support of $\cE$ is the integral hypersurface $X$, it follows that $\det(\varphi)$ is necessarily the form defining $X$ to the power $\rk(\cE)$.
\qed
\medbreak

Finally, we prove the uniqueness of sequence \eqref{seqSteinerInstanton}. 

\begin{proposition}
\label{pUnique}
Let $X\subseteq\p{n+1}$ with $n\ge 2$ be an integral hypersurface of degree $d$ and $\cO_X(h):=\cO_{X}\otimes\cO_{\p {n+1}}(1)$.

If $\cE$ and $\cE_0$ are locally Cohen--Macaulay $h$--linear sheaves on $X$ fitting into sequences
\begin{equation}
\label{seqMorphism}
\begin{gathered}
0\longrightarrow\cG(-1)\mapright\varphi\cF^\vee\mapright{q}\cE\longrightarrow0,\\
0\longrightarrow\cG_0(-1)\mapright{\varphi_0}\cF_0^\vee\mapright{q_0}\cE\longrightarrow0
\end{gathered}
\end{equation}
where $\cG$, $\cF$, $\cG_0$, $\cF_0$ are Steiner bundles on $\p{n+1}$, then each morphism $\kappa\colon \cE\to\cE_0$ can be lifted in a unique way to a commutative diagram
\begin{equation}
\label{DiagramComplete}
\begin{CD}
0@>>>\cG(-1)@>\varphi>>   \cF^\vee@> q>> \cE@>>>0\ \ \\
@.@V \vartheta(-1) VV @V \psi VV @V \kappa VV \\
0@>>>\cG_0(-1)@>{\varphi_0}>>   \cF_0^\vee@>q_0>> \cE_0@>>>0\ .\\
\end{CD}
\end{equation}
Moreover, if $\kappa$ is an isomorphism, then the same is true for both $\psi$ and $\vartheta$.
\end{proposition}
\begin{proof}
By applying $\Hom_{\p{n+1}}\big(\cF^\vee,\cdot\big)$ to the second sequence \eqref{seqMorphism}, we obtain the exact sequence
\begin{align*}
H^0\big(\cF\otimes\cG_0(-1)\big)\longrightarrow H^0\big(\cF\otimes\cF_0^\vee\big)\longrightarrow H^0\big(\cF\otimes\cE_0\big)\longrightarrow H^1\big(\cF\otimes\cG_0(-1)\big).
\end{align*}
Recall that $\cF$ is a Steiner bundle on $\p{n+1}$, hence the cohomology of sequence \eqref{seqSteinerSpace} tensored by $\cG_0$ returns
\begin{align*}
h^i\big(\cF\otimes\cG_0(-1)\big)\le h^0\big(\cF\big)h^i\big(\cG_0(-1)\big)+h^{n}\big(\cF(-n-1)\big)h^{i+1}\big(\cG_0(-2)\big)
\end{align*}
Thus $h^i\big(\cF\otimes\cG_0(-1)\big)=0$ for $0\le i\le 1$, because $\cG_0$ is Steiner on $\p{n+1}$, whence $h^i\big(\cG_0(-t)\big)=0$ for $0\le i,t\le 2$. 

We deduce that $\kappa\colon\cE\to \cE_0$ can be lifted in a unique way to $\psi\colon \cF^\vee\to\cF_0^\vee$. 
Since $0=q\circ\varphi=q_0\circ\psi\circ\varphi$, it follows that $\im(\psi\circ\varphi)\subseteq\ker(q_0)=\im(\varphi_0)\cong\cG_0(-1)$, hence we have a uniquely determined induced morphism $\vartheta(-1)\colon \cG(-1)\to \cG_0(-1)$.

If $\kappa$ is an isomorphism, then $\kappa^{-1}$ can be also lifted uniquely to another morphism $\psi_0\colon \cF_0^\vee\to\cF^\vee$. Notice that $\psi\circ\psi_0\colon \cF_0^\vee\to\cF_0^\vee$ and $\psi_0\circ\psi\colon\cF^\vee\to\cF^\vee$ are liftings of the identities of $\cE_0$ and $\cE$ respectively. Since the identities on $\cF_0$ and $\cF$ are also liftings of such maps, it follows that the same argument as above also implies that  $\psi$ and $\psi_0$ are inverse each other, i.e. $\psi$ is an isomorphism. It follows that also $\vartheta_0(-1)$ is an isomorphisomorphismism, thanks to the Five's lemma.
\end{proof}

\section{Steiner representation of surfaces in $\p3$}
\label{sDeterminantal}

As pointed out in the Introduction, there are no $h$--instanton line bundles on smooth hypersurfaces $X\subset\p{n+1}$ with $n\ge3$ thanks to \cite[Proposition 8.2]{An--Ca1}. Because of this, one can ask for Steiner representations only for smooth surfaces in $\p3$. In this case we are able to prove Theorem \ref{tDeterminantal} which generalizes \cite[Proposition 6.2]{Bea}.

\medbreak
\noindent{\it Proof of Theorem \ref{tDeterminantal}.}
Let $\cO_X(h):=\cO_X\otimes\cO_{\p 3}(1)$ and assume that $X$ contains a curve $C$ as in the statement. If $\cE:=\cO_X(C-\quantum h)$, then the value of $\deg(C)$ in the statement and the definition of $\cE$ return
$$
2c_1(\cE)h=d(d-1)=(3h+K_X)h.
$$
The value of $p_a(C)$ in the statement and the adjunction formula on $X$ return
$$
C^2=d\quantum^2 +\quantum(d^2-d-2)+{d\choose3},
$$
hence equality \eqref{RRsurface} yields $\chi(\cE(-h))=\quantum$. Finally
\begin{gather*}
h^0\big(\cE(-h)\big)=h^0\big(\cO_X(C-(\quantum+1)h)\big)=0,\\
h^2\big(\cE(-2h)\big)=h^2\big(\cO_X(C-(\quantum+2)h)\big)=h^0\big(\cO_X((d+\quantum-2)h-C)\big)=0.
\end{gather*}
Thus \cite[Theorem 1.6]{An--Ca1} yields that $\cE$ is a $h$--instanton line bundle on $X$ with quantum number $\quantum$ and the existence of the Steiner determinantal representation then follows from Corollary \ref{cHypersurfaceSpace}. Moreover, Corollary \ref{cHypersurface} implies that the size $s$ of such a representation is $d+2\quantum$

Conversely, assume that a determinantal representation $\varphi\colon\cG(-1)\to\cF^\vee$ of size $s=\rk(\cG)=\rk(\cF)$ exists. Thus $\det(\varphi)$ is the form defining $X$. Corollary \ref{cHypersurface} yields the existence of a locally Cohen--Macaulay $h$--instanton sheaf $\cE$ with $\rk(\cE)=1$, quantum number $\quantum=h^n\big(\cG(-n-1)\big)=h^n\big(\cF(-n-1)\big)$ and $s=d+2\quantum$. Since $X$ is smooth, it follows that $\cE$ is actually a line bundle as pointed out in the introduction. Thanks to Proposition \ref{pRegularity} we know that $\reg(\cE)\le \quantum$. In particular $\cE(\quantum h)$ is globally generated, hence the hypothesis on $\field$ and the Bertini theorem yield the existence of a smooth curve $C\subseteq X$ such that $\cO_X(C)\cong\cE(\quantum h)$. All the assertion of the statement then follow from such a definition of $C$ and the properties of instanton sheaves.
\qed
\medbreak

\begin{remark}
\label{rDeterminantal}
Equality \eqref{RRsurface} and adjunction formula on $X$ imply that equalities \eqref{InvariantsC} are equivalent to
\begin{gather*}
\deg(C)=\frac12{d(d+2\quantum-1)},\qquad C^2=d\quantum^2 +\quantum(d^2-d-2)+{d\choose3}.
\end{gather*}
This concludes the remark.
\end{remark}

\begin{remark}
\label{rSmoothCurve}
Notice that the smooth curve $C$ is necessarily connected if $d\ge3$, hence integral. Indeed we trivially have $h^0\big(\cO_X(-C)\big)=0$. Moreover, 
$$
h^1\big(\cO_X(-C)\big)=h^1\big(\cE^\vee(-\quantum h)\big)=h^1\big(\cE((\quantum+d-4) h)\big).
$$
Since $\quantum+d-4\ge\quantum-1$ and $\reg(\cE)\le \quantum$, the dimension on the right is zero. Thus the cohomology of sequence \eqref{seqStandard} with $Y:=C$ and $P:=X$ yields $h^0\big(\cO_C\big)=1$, hence $C$ is connected because it is smooth.

On the other hand, let $d=2$: thus $X\cong\p1\times\p1$ and we denote by $\ell_1$ and $\ell_2$ the standard generators of $\Pic(X)$, so that $h=\ell_1+\ell_2$. In this case the only $h$--instanton line bundles with quantum number $\quantum$ are $\cO_X((\quantum+1)\ell_i)$ for $i=1,2$. In particular, the smooth curve $C$ is not connected unless $\quantum=0$. This concludes the remark.
\end{remark}

In what follows we will inspect the existence of Steiner representations of smooth surfaces of low degree.

\begin{corollary}
Let $X\subseteq\p3$ be a smooth surface of degree $d=2,3$. Assume that the characteristic of $\field$ is $0$.

The surface $X$ has a Steiner determinantal representation of size $d+2\quantum$ for each $\quantum\ge0$.
\end{corollary}
\begin{proof}
If $d=2$, then we know from Remark \ref{rSmoothCurve} that $X$ supports two $h$--instanton line bundles with quantum number $\quantum$ for each $\quantum\ge0$, hence it has a Steiner determinantal representation of size $2+2\quantum$ (see the introduction).

Let $d=3$. Thanks to \cite[Corollaire 2.3]{Gr--Pe}, $X$ contains a smooth connected curve $C$  such that
$$
\deg(C)=3(\quantum+1),\qquad p_a(C)=\frac{\quantum}2(3\quantum+1).
$$
If $h^0\big(\cO_X(C-(\quantum+1)h)\big)\ne0$, then $C\in\vert(\quantum+1)h\vert$ because $\deg(C)=3(\quantum+1)$. In particular $C$ is the complete intersection of $X$ with a surface of degree $\quantum+1$, hence the adjunction formula on $\p3$ would imply
$$
p_a(C)=\frac{3\quantum}2(\quantum+1)+1,
$$
a contradiction. Thus we deduce $h^0\big(\cO_X(C-(\quantum+1)h)\big)=0$. A similar argument also yields $h^0\big(\cO_X((\quantum+1)h-C)\big)=0$. 

The statement then follows from Theorem \ref{tDeterminantal}.
\end{proof}

The case of smooth surfaces of degree $d\ge4$ is completely different and in the last part of this section we focus on the case $d=4$. 

\begin{corollary}
\label{cDeterminantal}
Let $X\subseteq\p3$ be a smooth quartic surface. Assume that the characteristic of $\field$ is $0$.

The surface $X$ has a Steiner determinantal representation of size $4+2\quantum$ if and only if it contains a smooth connected curve $C$ with
\begin{equation}
\label{InvariantCurve}
\deg(C)=4\quantum+6,\qquad p_a(C)=2\quantum^2+5\quantum+3,
\end{equation}
and,  when $\quantum=0,2$, such that  $h^0\big(\cO_X((\quantum+2)h-C)\big)=0$.
\end{corollary}
\begin{proof}
If $X$ has a Steiner determinantal representation of size $4+2\quantum$, then it contains a smooth curve $C$ satisfying equalities \eqref{InvariantCurve} and such that $h^0\big(\cO_X(C-(\quantum+1)h)\big)=h^0\big(\cO_X((\quantum+2)h-C)\big)=0$, thanks to Theorem \ref{tDeterminantal}. Moreover such a curve $C$ is also connected, hence irreducible, by Remark \ref{rSmoothCurve}.

Conversely, assume that $X$ contains a smooth connected curve $C$ satisfying equalities \eqref{InvariantCurve} and, when $\quantum=0,2$, also $h^0\big(\cO_X((\quantum+2)h-C)\big)=0$: notice that the latter equality is equivalent to $\vert (\quantum+2)h-C\vert=\emptyset$ because $C\not\in\vert (\quantum+2)h\vert$. 

We first claim that such a latter equality actually holds regardless of $\quantum\ge0$. Indeed, assume that $D\in\vert (\quantum+2)h-C\vert$. Then $\deg(D)=2$ and $p_a(D)=-\quantum-1$ by adjunction on $X$, hence $D$ is either the union of two disjoint lines or a double line. In the former case, $D^2=-4$, i.e. $p_a(D)=-1$ whence $\quantum=0$: in the latter case, $D^2=-8$, i.e. $p_a(D)=-3$ whence $\quantum=2$. Finally, regardless of the value of $\quantum$, if $D\in\vert C-(\quantum+1)h\vert$, then $DC=-2$, contradicting the irreducibility of $C$. Thus we also have $h^0\big(\cO_X(C-(\quantum+1)h)\big)=0$. The statement then follows from Theorem \ref{tDeterminantal}.
\end{proof}

In particular, when $\quantum=0$, we recover \cite[Corollary 6.6]{Bea}, i.e. $X$ has a determinantal representation if and only if it contains a smooth connected curve $C$ with $\deg(C)=6$, $C^2=4$, $p_a(C)=3$ not lying on a quadric surface. 

Now let us deal with  Steiner determinantal representations of size larger than $4$ of smooth quartics. As we already mentioned the picture is quite  different with respect to the cases of quadrics and cubics.

\begin{proposition}
\label{pQuarticDeterminantal}
Let $X\subseteq\p3$ be a very general smooth determinantal quartic surface over $\field=\bC$.

The surface $X$ has a Steiner determinantal representation of size $4+2\quantum$ if and only if  $\quantum=10\lambda(\lambda+1)$ for $\lambda\ge0$.
\end{proposition}
\begin{proof}
If $X\subseteq\p3$ is a determinantal smooth quartic surface, then its Picard lattice contains the even sublattice  $\Lambda:=\bZ h\oplus\bZ C_0$, where $\cO_X(h):=\cO_X\otimes\cO_{\p3}(1)$ and $C_0$ is a smooth connected curve such that $\deg(C_0)=6$, $C_0^2=4$ and $p_a(C_0)=3$ (see Corollary \ref{cDeterminantal}). Thanks to \cite[Theorem II.3.1]{Lop} the Picard lattice of the very general determinantal smooth quartic surface coincides with $\Lambda$.

If $X$ has a Steiner determinantal representation of size $4+2\quantum$, then we can find a smooth connected curve $C\subseteq X$ satisfying equalities \eqref{InvariantCurve} (see also Remark \ref{rDeterminantal}). In particular, if $C$ is linearly equivalent to $xh+yC_0$, then
$$
4\quantum+6=4x+6y,\qquad 4\quantum^2+10\quantum+4=4x^2+12xy+4y^2.
$$
Thus $\quantum=5(y^2-1)/2$: since $y\in\bZ$, it follows that $y=2\lambda+1$, hence $\quantum=10\lambda(\lambda+1)$ (and $x=\lambda(10\lambda+7)$).

Conversely, let $\quantum=10\lambda(\lambda+1)$ for $\lambda\ge0$: notice that $\quantum\ne2$ and $\quantum=0$ if and only if $C=C_0$. Moreover, $\cO_X(C_0)$ is globally generated (see \cite[Corollary 3.7]{Cas}), hence $\cO_X(\lambda(10\lambda+7)h+(2\lambda+1)C_0)$ is very ample (see \cite[Exercise II.7.5 (d)]{Ha2}). It follows that the general element $C\in \vert \lambda(10\lambda+7)h+(2\lambda+1)C_0\vert$ is a smooth connected curve satisfying equalities \eqref{InvariantCurve}. Corollary \ref{cDeterminantal} then yields the statement.
\end{proof}

\begin{proposition}
\label{pExistenceQuartic}
Let $\field=\bC$.

If $\quantum\ge0$, there exists a smooth quartic surface $X\subseteq\p3$ satisfying the following properties.
\begin{enumerate}
\item $X$ has a Steiner determinantal representation of size $4+2\quantum$.
\item $X$ has not any Steiner determinantal representation of size $4+2\widehat{\quantum}$ for each $\widehat{\quantum}< \quantum$.
\end{enumerate}
\end{proposition}
\begin{proof}
Thanks to \cite[Theorem 1.1]{Knu2} there exists a smooth quartic surface $X\subseteq\p3$ whose Picard lattice is  $\Lambda:=\bZ h\oplus\bZ C$, where $\cO_X(h):=\cO_X\otimes\cO_{\p3}(1)$ and $C$ is a smooth curve satisfying equalities \eqref{InvariantCurve}.

In order to show the existence of a Steiner determinantal representation of size $4+2\quantum$ for $X$, it suffices to check that $h^0\big(\cO_X((\quantum+2)h-C)\big)=0$ which is equivalent to $\vert \cO_X((k+2)h-C)\vert=\emptyset$ because $C\not\in \vert (k+2)h\vert$. If $E\in \vert \cO_X((k+2)h-C)\vert$, then $Eh=2$ and $E^2=-2\quantum-4$, hence the support of $E$ contains a line $L$. If $L=xh+yC$, then $1=Lh=4x+(4\quantum+6)y$, a contradiction.

Assume that $X$ has a Steiner determinantal representation of size $4+2\widehat{\quantum}$. Thus $X$ would contain a curve $\widehat{C}$ such that $\widehat{C}h=4\widehat{\quantum}+6$ and $\widehat{C}^2=4\widehat{\quantum}^2+10\widehat{\quantum}+4$. If $\widehat{C}\in\vert \cO_X(xh+yC)\vert$, then
$$
\widehat{C}h=4x+ (4\quantum+6)y, \qquad \widehat{C}^2=4x^2+2(4\quantum+6)xy+(4\quantum^2+10\quantum+4)y^2.
$$
Simple computations then lead to the equality $(2\quantum+5)y^2=2\widehat{\quantum}+5$ which has no integer solutions if $\quantum>\widehat{\quantum}$.
\end{proof}

The Hilbert scheme of smooth quartics in $\p3$ is identified with an open non--empty subset $\fK\subseteq\vert\cO_{\p3}(4)\vert\cong\p{34}$. The locus of determinantal quartics $\fK_0\subseteq\fK$ has dimension
$$
16h^0\big(\cO_{\p3}(1)\big)-1-2\dim(\PGL(\p3))=33.
$$
Proposition \ref{pExistenceQuartic} guarantees that the locus $\fK_\quantum\subseteq\fK$ of surfaces having a Steiner determinantal representation of size $4+2\quantum$ is not contained in $\fK_0$ for each $\quantum\ge1$. Proposition \ref{pQuarticDeterminantal} also yields that $\fK_0\not\subseteq\fK_\quantum$ for infinitely many $\quantum\ge1$. When $\quantum=1$ we can  say something more about the locus $\fK_1$.

\begin{proposition}
\label{pHilbertQuartic}
Let $\field=\bC$.

The locus $\fK_1\subseteq\fK$ has dimension $33$.
\end{proposition}
\begin{proof}
Let $\fH$ be the union of the components of the Hilbert scheme containing curves $C\subseteq\p3$ such that $\deg(C)=p_a(C)=10$ and  consider the incidence relation 
$$
\overline{\fX}:=\left\{\ (X,C)\ \vert\ C\subseteq X\ \right\}\subseteq \fK\times\fH.
$$
Proposition \ref{pExistenceQuartic} implies that the open subset $\fX$ of pairs $(X,C)$ such that
$$
h^0\big(\cO_X(3h-C)\big)=0
$$
is non--empty. We have two projections $\alpha\colon \fX\to\fH$ and $\beta\colon \fX\to\fK$. 

On the one hand, the locus of quartic surfaces admitting a Steiner determinantal representation of size $6$ is $\fK_1:=\im(\beta)\subseteq\fK$. Since the very general quartic surface has cyclic Picard group, it follows that each irreducible component $\fK_1'\subseteq\fK_1$ has dimension $33$ at most. Equality \eqref{RRsurface} yields that the fibre of $\beta$ over $X\in\fK_1$ is isomorphic to $\p{10}$. Taking $\fK_1'$ as any component of maximal dimension we deduce that
\begin{equation}
\label{fK}
\dim(\fX)=\dim(\beta^{-1}(\fK_1'))=\dim(\fK_1')+10.
\end{equation}

On the other hand, if $C\in\fH$, then sequence \eqref{seqStandard} with $Y:=C$ and $P:=\p3$ tensored by $\cO_{\p3}(4)$ yields $\dim(\alpha^{-1}(C))=3+h^1\big(\cI_{C\vert\p3}(4)\big)$, thanks to equality \eqref{RRcurve}. Consider any component $\fH'$ over which $h^1\big(\cI_{C\vert\p3}(4)\big)$ is minimal.

We know that $\omega_C\cong\cO_C\otimes\cO_X(C)$ thanks to the adjunction formula on $X$. Thus, for each pair $(X,C)\in\fX$, the cohomology of the exact sequence
$$
0\longrightarrow\cN_{C\vert X} \longrightarrow\cN_{C\vert \p3} \longrightarrow\cO_C\otimes\cN_{X\vert \p3} \longrightarrow0
$$
and the isomorphisms $\cN_{X\vert \p3}\cong\cO_X\otimes\cO_{\p3}(4)$, $\cN_{C\vert X}\cong\cO_C\otimes\cO_X(C)\cong\omega_C$ yield 
\begin{equation}
\label{fH}
43+h^1\big(\cI_{C\vert\p3}(4)\big)\le\dim(\fH')+\dim(\alpha^{-1}(C))\le \dim(\fX).
\end{equation}
By combining equality \eqref{fK} with inequality \eqref{fH} we finally obtain 
$$
33+h^1\big(\cI_{C\vert\p3}(4)\big)\le\dim(\fK_1'),
$$
whence $\dim(\fK_1)=\dim(\fK_1')=33$.
\end{proof}

\begin{remark}
Using the notation introduced in the proof of Proposition \ref{pHilbertQuartic} we also deduce that $\dim(\fH')=40$ and $h^1\big(\cI_{C\vert\p3}(4)\big)=0$ for $C\in\fH'$.
\end{remark}

Thus the first actually interesting case is when $\rk(\cE)=2$ and we will inspect it in the following sections.

\section{Steiner pfaffian representation of hypersurfaces in $\p{n+1}$}
\label{sPfaffian}
We are mostly interested to Steiner representations coming from instanton bundles of rank greater than $1$. In order to deal with some examples in this direction we start by proving Theorem \ref{tPfaffian} stated in the introduction. Its proof is similar to the one of   \cite[Theorem B]{Bea}.

\medbreak
\noindent{\it Proof of Theorem \ref{tPfaffian}.}
Thanks to Corollaries \ref{cHypersurface} and \ref{cHypersurfaceSpace} we know the existence of sequence \eqref{seqSteinerInstanton} with $\rk(\cG)=\rk(\cF)=\rk(\cE)d+2\quantum$.

In what follows we will prove that $\cG\cong\cF$ and that $\varphi$ can be actually taken $\varepsilon$--symmetric, i.e. such that $\varphi^\vee(-1)=\varepsilon\varphi$. Let $F$ be a degree $d$ irreducible form defining $X$. Since $\cE$ is supported on $X$, it follows that the covariant functor $\sHom_{\p{n+1}}\big(\cE,\cdot\big)$ applied to the exact sequence
$$
0\longrightarrow\cO_{\p {n+1}}(-d)\mapright{\cdot F}\cO_{\p {n+1}}\longrightarrow\cO_X\longrightarrow0
$$
yields the isomorphisms
$$
\cE^\vee:=\sHom_X\big(\cE,\cO_X\big)\cong\sHom_{\p {n+1}}\big(\cE,\cO_X\big)\cong\sExt^1_{\p {n+1}}\big(\cE,\cO_{\p {n+1}}(-d)\big)
$$
because the multiplication by $F$ induces the zero map. 

By applying $\sHom_{\p {n+1}}\big(\cdot,\cO_{\p{n+1}}(-1)\big)$ to sequence \eqref{seqSteinerInstanton} we obtain an exact sequence
\begin{equation}
\label{seqSteinerInstantonDual}
0\longrightarrow\cF(-1)\mapright{\varphi^\vee(-1)}\cG^\vee\mapright{q'}\cE^\vee((d-1)h)\longrightarrow0
\end{equation}

By definition there is an isomorphism $\eta\colon\cE\mapright\sim\cE^\vee((d-1)h)$, hence Proposition \ref{pUnique} yields the existence of a commutative 
\begin{equation*}
\begin{CD}
0@>>>\cG(-1)@>\varphi>>   \cF^\vee@>q>> \cE@>>>0\ \ \\
@.@V \vartheta(-1) VV @V \psi VV @V \eta VV \\
0@>>>\cF(-1)@>\varphi^\vee(-1)>>   \cG^\vee@>q'>> \cE^\vee((d-1)h)@>>>0\ .\\
\end{CD}
\end{equation*}
where all the vertical arrows are isomorphisms. By applying $\sHom_{\p{n+1}}(\cdot,\cO_{\p{n+1}}(-1)\big)$ to the above diagram, taking into account that $\cE$ is reflexive and $\eta^\vee(d-1)=\varepsilon \eta$, we obtain the second commutative diagram
\begin{equation*}
\begin{CD}
0@>>>\cG(-1)@>\varphi>>   \cF^\vee@>q>> \cE@>>>0\ \ \\
@.@V \psi^\vee(-1) VV @V \vartheta^\vee VV @V \eta^\vee(d-1) VV \\
0@>>>\cF(-1)@>\varphi^\vee(-1)>>   \cG^\vee@>q'>> \cE^\vee((d-1)h)@>>>0\ .\\
\end{CD}
\end{equation*}
Since $\varepsilon\eta^\vee(d-1)=\eta$, it follows that both $\psi$ and $\varepsilon\vartheta^\vee$ are liftings of $\varepsilon\eta^\vee(d-1)=\eta$, hence they must coincide by Proposition \ref{pUnique}. In particular
$$
\psi^{-1}\circ\varphi^\vee(-1)=\varphi\circ\vartheta^{-1}(-1)=\varphi\circ(\varepsilon\psi^{-1})^\vee(-1)=\varepsilon(\psi^{-1}\circ\varphi^\vee(-1))^\vee.
$$
Replacing $\varphi$ with $\psi^{-1}\circ\varphi^\vee(-1)$ we can assume that $\varphi$ is $\varepsilon$--symmetric and  the proof of the statement is complete.
\qed
\medbreak

The most common cases of $\varepsilon$--orientable sheaves occur when $X$ is smooth. In this case a locally Cohen--Macaulay $h$--instanton sheaf $\cE$ on $X$ is locally free thanks to the Auslander–Buchsbaum formula \cite[Theorem 1.3.3]{Br--He}. If $\rk(\cE)=2$, we can consider the skew--symmetric morphism induced by the exterior product so that $\varepsilon=-1$. 

By Theorem \ref{tPfaffian} if a rank two orientable $h$--instanton sheaf $\cE$ with quantum number $\quantum$ on an integral hypersurface $X\subseteq\p n$ of degree $d$ exists, then there is a Steiner pfaffian representation of $X$ of size $2(d+\quantum)$, i.e. there is a Steiner bundle $\cF$ on $\p n$ with $h^{n}\big(\cF(-n-1)\big)=\quantum$ and rank $r=2(d+\quantum)$ and a skew--symmetric morphism $\varphi\colon\cF(-1)\to\cF^\vee$ such that $X=D_{r-2}(\varphi)$.

In this case the pfaffian $\pf(\varphi)$ of $\varphi$ is defined. Moreover, $\det(\varphi)=\pf(\varphi)^2$ and $D_{r-1}(\varphi)=D_{r-2}(\varphi)=\{\ \pf(\varphi)=0\ \}$. 

If $X$ is smooth then an $\varepsilon$--orientable $h$--instanton sheaf with quantum number $\quantum=0$ is actually an Ulrich bundle. In this particular case $\cE$ is $h$--Ulrich and the same is true for $\cF$ hence $\cF\cong \cO_{\p{n+1}}^{\oplus2d}$ by the Horrocks theorem. Thus the polynomial defining $X$ is the pfaffian of a $2d\times2d$ skew--symmetric matrix of linear forms. 

\begin{remark}
Let $n\ge 3$. We confront the above Steiner pfaffian representation of smooth hypersurfaces $X\subseteq\p{n+1}$ with the bundle pfaffian representation described in \cite[Subsection 2.1]{Ka--Ka} based on the results in \cite{Wal}.

Consider a hypersurface $X\subseteq\p{n+1}$ endowed with a rank two orientable $h$--instanton bundle $\cE$ with quantum number $\quantum$: thanks to Theorem \ref{tPfaffian} we know the existence of a skew--symmetric morphism
$$
\varphi\colon\cF(-1)\longrightarrow\cF^\vee
$$
where $\cF$ is a Steiner bundle with $h^{n}\big(\cF(-n-1)\big)=\quantum$ and $X$ is defined by $\pf(\varphi)$

Consider any non--negative integer $\ell$ such that $\cE(\ell h)$ has a section vanishing on a codimension two subscheme $Y\subseteq X$. On the one hand, since 
$$
\det(\cN_{Y\vert X})\cong\cO_Y\otimes\det(\cE(\ell h))\cong\cO_Y\otimes \cO_X((d+2\ell-1)h), 
$$
it follows from Theorem \ref{tSerre} that $\cI_{Y\vert X}$ fits into the  sequence \eqref{seqSerre} with $Z:=Y$,  $P:=X$ and $\cA:=\cE(\ell h)$. On the other hand, adjunction formula on $X$ yields $\omega_Y\cong\cO_Y\otimes\cO_{\p{n+1}}(2d+2\ell-n-3)$. In particular, if $n$ is even, then $2d+2\ell-n-3$ is odd, hence \cite[Theorem 0.1]{Wal} implies the existence of a vector bundle $\cU$ on $\p{n+1}$ such that $\cI_{Y\vert\p{n+1}}$ fits into the exact sequence
$$
0\longrightarrow\cO_{\p{n+1}}(-2s-t)\longrightarrow\cU(-s-t)\mapright{\theta}\cU^\vee(-s)\longrightarrow\cI_{Y\vert\p{n+1}}\longrightarrow0
$$
where $\theta$ is skew--symmetric and $s,t,p$ are such that $\rk(\cU)=2p+1$, $c_1(\cU)=pt-s$ and $\omega_Y\cong\cO_Y\otimes\cO_{\p{n+1}}(2s+t-n-2)$, hence $t=2(d+\ell-s-1)+1$. 

If we make the additional assumption $H^1_*\big(\cU\big)=0$, setting $\cV:=\cU(s+1-d-\ell)$, then there is another  skew--symmetric morphism
$$
v\colon(\cV\oplus\cO_{\p{n+1}}(\ell))(-1)\longrightarrow(\cV\oplus\cO_{\p{n+1}}(\ell))^\vee
$$
such that $\pf(v)$ is a form defining $X$, as pointed out in \cite[Subsection 2.1]{Ka--Ka}. 

In order to confront the skew--symmetric morphisms $\varphi$ and $v$, we first notice that if $\ell\ne0$, then $\cV\oplus\cO_{\p{n+1}}(\ell)$ is certainly not Steiner. While in the case $\ell=0$ we obtain the strong restriction $h^0\big(\cE\big)\ne0$.
\end{remark}

The following result is a by--product of Theorem \ref{tPfaffian} and of some results proved in \cite[Example 6.1]{An--Ca1}: it generalizes \cite[Proposition 7.6]{Bea}.

\begin{proposition}
\label{pPfaffianSurface}
Let $X\subseteq\p3$ be a smooth surface of degree $d$ over $\field=\bC$.

There exists a Steiner pfaffian representation of $X$ of even size $r$ such that
\begin{equation*}
2d\le r\le r(d):=\left\lbrace\begin{array}{ll} 
2d \quad&\text{if $d\le 4$,}\\
\dfrac13\left(10d^3-39d^2+35d+12\right)\quad&\text{if $d\ge5$.}
\end{array}\right.
\end{equation*}
\end{proposition}
\begin{proof}
Let $\cO_X(h):=\cO_X\otimes\cO_{\p3}(1)$. Notice that $K_X=(d-4)h$, $h^2=d$, $q(X)=0$ and $p_g(X)={{d-1}\choose3}$.

The polynomial defining each surface $X\subseteq\p3$ of degree $d\le 4$ is actually the pfaffian of a $2d\times2d$ skew--symmetric matrix with linear entries (see \cite[Proposition 7.6 (a)]{Bea} and \cite[Corollary 1.2]{C--K--M}), hence the statement is true in this case. 

If $d\ge5$, as explained in \cite[Example 6.10]{An--Ca1}, we can construct a rank two orientable $h$--instanton bundle $\cE$ on $X$ as follows. Let $C\in\vert(d-1)h\vert$ be smooth and connected, and  $\cO_C(D)$ globally generated such that $h^1\big(\cO_X((d-3)h)\otimes\cO_C(D)\big)=0$. Thus there is an exact sequence
$$
0\longrightarrow\cE^\vee\longrightarrow\cO_X^2\mapright{(\sigma_0,\sigma_1)}\cO_C(D)\longrightarrow0.
$$
where $\sigma_0,\sigma_1\in H^0\big(\cO_C(D)\big)$ have no common zeros. The sheaf $\cE$ is a rank two orientable $h$--instanton bundle with $c_1(\cE)=(d-1)h$ and $c_2(\cE)=\deg(D)$, hence
\begin{align*}
\quantum=h^1\big(\cE(-h)\big)=\deg(D)-\frac d6\left(2d^2-3d+1\right).
\end{align*}
thanks to equality \eqref{RRsurface}. Thus we have to estimate $\deg(D)$. By adjunction we have $\cO_C(K_C)\cong\cO_C\otimes\cO_X(3h+2K_X)$, hence every divisor $D$ with
$$
\deg(D)=2p_a(C)=(3h+2K_X)C+2=2d^3-7d^2+5d+2
$$
is globally generated. If this is the case, then we also have $\deg(D)> (d-2)hC$, hence
$$
h^1\big(\cO_X((d-3)h)\otimes\cO_C(D)\big)=h^0\big(\cO_X((d-2)h)\otimes\cO_C(-D)\big)=0.
$$
It follows the existence of a rank two orientable $h$--instanton bundle $\cE$ on $X$ with
$$
\quantum=\frac16\left(10d^3-39d^2+29d+12\right).
$$
Thus the statement follows from Theorem \ref{tPfaffian}, because $r=\rk(\cF)=2d+2\quantum$, where $\quantum\ge0$ by definition.
\end{proof}

The argument used in the proof above for computing $r(d)$ when $d\ge5$ works for $d=4$ as well and it yields the upper bound $60$ in this case, which is very far from the result in \cite{C--K--M}: thus the function $r(d)$ is probably far from being optimal for each $d\ge4$. 

On the one hand, as we pointed out in the proof, the form defining a general surface $X$ of degree $d\le 15$ is the pfaffian of a $2d\times2d$ skew--symmetric matrix with linear entries. 
On the other hand,  when $d\geq 16$ the known best upper bound for the size of a square matrix of linear forms whose determinant is a power of the form defining $X$ is  $d^{\mathrm{ch}(f)-1}$ where 
$$
\mathrm{ch}(f)=\left\lceil \frac{1}{3d+1}{{d+3}\choose 3}\right\rceil:
$$
see \cite[Theorem 4.2.14]{C--MR--PL} and \cite[Corollary 1.3]{To--Vi}. Notice that $d^{\mathrm{ch}(f)-1}$ is considerably larger than the function $r(d)$ in Proposition \ref{pPfaffianSurface}.

Moreover, if $d\ge16$ we certainly know the size $r$ of a Steiner pfaffian representation for the general $X$ is at least $2d+2$, thanks to \cite[Proposition 7.6]{Bea}. 
 
Thus the following question is natural.

\begin{question}
Which are the sharpest lower and upper bounds for $r$? 
\end{question}

If  $\cF$ is a vector bundle of even rank $r$ on $\p {n+1}$ and $\varphi\in H^0\big((\wedge^2\cF^\vee)(1)\big)$ is an injective skew--symmetric morphism, then it is expected that $D_{r-4}(\varphi)\ne\emptyset$ if $n\ge5$. In particular, $D_{r-2}(\varphi)$ turns out to be singular in that range. Thus the following question also sounds natural. 

\begin{question}
Which classes of smooth hypersurfaces in $\p4$ and $\p5$ are Steiner pfaffian?
\end{question}

\section{Further examples of Steiner pfaffian representations}
\label{sExample}
In this section we collect some results and examples about the Steiner pfaffian representation of smooth hypersurfaces $X\subseteq \p {n+1}$ for low values of $n\ge 3$. For simplicity, we will implicitly assume in all the examples that $\field=\bC$ and we set $\cO_X(h):=\cO_X\otimes\cO_{\p {n+1}}(1)$.

Let us first examine some examples in $\p {n+1}$ with $n\ge4$.

\begin{example}
\label{eQuadric}
Let $X\subseteq\p {n+1}$ be a smooth quadric hypersurface when $n\ge4$.

If $n\ge6$, then there are no rank two $h$--instanton bundles $X$ (see \cite[Proposition 6.3]{An--Ca1}). Nevertheless, each such $X$ supports $h$--Ulrich bundles of rank $2^{\left[\frac{n-1}2\right]}$, namely the spinor bundles, and there are no other indecomposable aCM bundles of rank greater than $1$ up to shift (e.g. see \cite{Ei--He}).

If $n=5$, then $X$ supports the Cayley bundle $\cE$ (see \cite{Ott3}) which is a rank two orientable $h$--instanton. Since $h^0\big(\cE\big)=h^1\big(\cE\big)=0$ and $\quantum=h^1\big(\cE(-h)\big)=1$ (see \cite[Theorem 3.1]{Ott3}), it follows that sequence \eqref{seqPfaffian} becomes
$$
0\longrightarrow(\Omega_{\p6}^1)^\vee(-2)\mapright\varphi\Omega_{\p6}^1(1)\longrightarrow\cE\longrightarrow0
$$
where $\varphi$ is skew symmetric. Notice that
$$
h^0\big(\Omega^2_{\p 6}(3)\big)=35, \qquad h^0\big(\cO_{\p 6}(2)\big)=28,\qquad  \dim(\Aut(\Omega_{\p6}^1(1)))=1,
$$
which is consistent with the fact that the moduli space of Cayley bundles on a fixed quadric surface has dimension 7 (see \cite[Main Theorem]{Ott3}). In particular, $X$ has a Steiner pfaffian representation of size $6$, but not of size $4$ (because it does not support any Ulrich bundle of rank of rank two)

Let $n=4$. In this case in \cite{Ar--So} a rank two orientable $h$--instanton $\cE$ on $X$ is described. We know that $h^0\big(\cE\big)=1$, $h^1\big(\cE\big)=0$ and $\quantum=h^1\big(\cE(-h)\big)=1$  (see \cite[First table at p. 207]{Ar--So}), hence sequence \eqref{seqPfaffian} becomes
$$
0\longrightarrow(\Omega_{\p5}^1)^\vee(-2)\oplus\cO_{\p 5}(-1)\mapright\varphi\Omega_{\p5}^1(1)\oplus\cO_{\p 5}\longrightarrow\cE\longrightarrow0
$$
where $\varphi$ is skew symmetric. Notice that the latter sequence can be obtained by restricting the former one to a hyperplane $\p5\cong\Sigma\subseteq\p6$. If we set $\Omega:=\Omega_{\p5}^1(1)\oplus\cO_{\p 5}$, then 
$$
h^0\big((\wedge^2\Omega)(1)\big)=35,\qquad h^0\big(\cO_{\p 5}(2)\big)=21,\qquad \dim(\Aut(\Omega))=8,
$$
which is again consistent with the fact that the moduli space of the bundles above has dimension 7 (see \cite[Theorem p. 206]{Ar--So}). In particular, $X$ has a Steiner pfaffian representation of size $6$. Moreover, it has also Steiner pfaffian representations of size $4$ coming from the two aforementioned rank two spinor bundles.
\end{example}

\begin{example}
\label{eCubic4}
In this example we deal with smooth cubic hypersurfaces $X\subseteq\p 5$. The moduli space $\cC$ of smooth cubic hypersurfaces in $\p5$ is birationally isomorphic to $H^0\big(\cO_{\p 5}(3)\big)/GL_6$, hence $\dim(\cC)=20$. 
 
In $\cC$ there is an interesting divisor $\cC_{18}\subseteq\cC$ which has been recently studied in \cite{A--H--T--VA}. We do not give here a precise definition of $\cC_{18}$, but we only mention that its general point represents a hypersurface $X$ containing an isomorphic projection ${Y_6}\subseteq\p5$ of a sextic del Pezzo surface $D_6\subseteq\p6$: for details see \cite[Proposition 6 and its proof]{A--H--T--VA}.

By definition we have  $\omega_{Y_6}\cong\cO_{Y_6}\otimes\cO_{\p5}(-1)$. Moreover,  ${Y_6}$ is non--degenerate, hence $h^0\big(\cI_{{Y_6}\vert X}(h)\big)=0$. Sequence \eqref{seqStandard} with $Y:={Y_6}$ and $P:=X$ yields $h^1\big(\cI_{{Y_6}\vert X}\big)=0$, $h^1\big(\cI_{{Y_6}\vert X}(h)\big)=1$ and $h^2\big(\cI_{{Y_6}\vert X}(-h)\big)=h^1\big(\omega_{Y_6}\big)=0$.
Adjunction formula on $X$ yields $\det(\cN_{{Y_6}\vert X})\cong \cO_{Y_6}\otimes\cO_{\p5}(2)$, hence sequence \eqref{seqSerre} takes the form
$$
0\longrightarrow\cO_X\longrightarrow\cE\longrightarrow\cI_{{Y_6}\vert X}(2h)\longrightarrow0
$$
where $\cO_X(h):=\cO_X\otimes\cO_{\p5}(1)$. Computing the cohomology of the twists of $\cE$ from the above sequence we deduce $h^i\big(\cE(-(i+1)h)\big)=0$ for $i\le 3$ and $h^1\big(\cE(-h)\big)=1$, hence $\cE$ is a rank two orientable $h$--instanton on $X$ with quantum number $\quantum=1$, thanks to \cite[Proposition 6.7]{An--Ca1} because $c_1(\cE)=2h$. In particular $X$ has a Steiner pfaffian representation of size $8$.

Let $\cC_{14}^\circ\subseteq\cC$ be the locus of hypersurfaces containing a quintic del Pezzo surface (see \cite{B--R--S} for details on $\cC_{14}^\circ$ and its closure 
$\cC_{14}\subseteq\cC$: see also \cite[Section 8]{Bea}). As in the previous case one easily checks that fourfolds represented by points of $\cC_{14}^\circ$ support a rank two orientable Ulrich bundle, hence they are pfaffian  in the usual sense. The locus $\cC_{14}$ is a divisor. 

Finally, let us consider any del Pezzo surfaces $D_m\subseteq \p m$ of degree $7\le m\le 9$, $Y_8\not\cong\p1\times\p1$. One can consider a linear space $H\subseteq\p \delta$ of dimension $m-6$ not intersecting its secant variety $\sigma_2(D_m)$. Thus the projection from $H$ onto $\p5$ maps $D_m$ isomorphically onto a surface $Y_m$. As pointed out in Appendix \ref{sMacaulay} using the software {\em Macaulay2} (see \cite{Gr--St}), one checks that for a general $H$ the homogeneous ideal of $Y_m$ in $\p 5$ contains two linearly independent smooth cubics if $m\ne9$. When $m=9$ there still exist two linearly independent smooth cubics through $Y_9$ for a suitable choice of $H$ as proved in \cite[Proposition 3.1]{Ka--Ka}. 

Thus the same argument used in the case $m=6$ shows the existence of smooth cubic hypersurfaces $X\subseteq\p 5$ supporting orientable $h$--instanton bundles with quantum number $\quantum=m-5$, hence with Steiner pfaffian representations of size $2m-4$.
\end{example}

\begin{example}
\label{eQuartic}
The construction of rank two orientable instanton bundles on some cubic hypersurfaces in $\p5$ described in Example \ref{eCubic4} is related to an analogous construction on certain quartic hypersurfaces in $\p5$ via the method described in \cite{Ka--Ka}

More precisely, as in the proof of \cite[Theorem 3.5]{Ka--Ka}, one first links $Y_m$ to another surface $Z_m$ of degree $36-m$ through a complete intersection of two smooth cubics and a quartic. Then one links such a $Z_m$ to a smooth surface $S_m$ of degree $m+9$ through a complete intersection of two smooth cubics and a quintic. The surface $S_m\subseteq\p5$ is non--degenerate, regular and canonically embedded, hence minimal: in particular $h^0\big(\cO_S\otimes\cO_{\p5}(2)\big)=m+16$ by \cite[Proposition VII.5.3]{B--H--P--VV}.

Thus \cite[Theorem 5.3.1]{Mi} implies 
\begin{equation}
\label{Schenzel1}
h^1\big(\cI_{S_m\vert \p5}(2h)\big)=h^1\big(\cI_{Y_m\vert \p5}(h)\big)=m-5,
\end{equation}
for each smooth quartic hypersurface $X$ through $S_m$. The cohomology of sequence \eqref{seqStandard} with $Y:=S_m$ and $P:=\p5$ returns
\begin{equation}
\label{Schenzel2}
h^0\big(\cI_{S_m\vert \p5}(2)\big)=0.
\end{equation}
Theorem \ref{tSerre} and adjunction on $X$ imply the existence of a rank two bundle $\cE$ with $\det(\cE)=\cO_X(3h)$ fitting into sequence \eqref{seqSerre} with $X:=S_m$ and $P:=X$.
On the one hand, $h^0\big(\cE(-th)\big)=h^0\big(\cI_{S_m\vert X}((3-t)h)\big)$. On the other hand, the exact sequence
$$
0\longrightarrow\cI_{X\vert \p5}\longrightarrow\cI_{S_m\vert \p5}\longrightarrow\cI_{S_m\vert X}\longrightarrow0
$$
and the isomorphism $\cI_{X\vert \p5}\cong\cO_{\p5}(-4)$ imply
$$
h^i\big(\cI_{S_m\vert \p5}(3-t)\big)=h^i\big(\cI_{S_m\vert X}((3-t)h)\big).
$$
Equalities \eqref{Schenzel1} and \eqref{Schenzel2} yield $h^0\big(\cE(-h)\big)=0$, $h^1\big(\cE(-h)\big)=m-5$: moreover  $h^0\big(\cE(-2h)\big)=0$ because the embedding $S_m\subseteq\p5$ is non--degenerate.

We conclude that $\cE$ is a rank two orientable $h$--instanton bundle with quantum number $\quantum=m-5$ on the quartic $X$, which then has a Steiner pfaffian representations of size $2m-2$.
\end{example}

\begin{example}
\label{eCubicHigher}
We give some further examples of smooth cubic hypersurfaces $X\subseteq\p {n+1}$ with $n\ge5$.

Let us consider the del Pezzo $(n-2)$--fold of degree $5$ in $\p {n+1}$ for $5\le n\le 8$. Its homogeneous ideal is generated by quadric forms, hence it is contained in many smooth cubic hypersurfaces. Again the same argument used in Example \eqref{eCubic4} shows that each such hypersurface $X$ supports a rank two orientable Ulrich bundle.

Finally, we consider the other del Pezzo $(n-2)$--folds $D_m\subseteq\p{n+m-4}$ of degree $m\ge6$. The postulated dimension of the secant variety $\sigma_2(D_m)$ is $2n-3$, hence $\dim(\sigma_2(D_m))=2n-3-\delta$ where $\delta\ge0$ is the defect of $D_m$. In particular $D_m$ can be isomorphically projected to $Y_m\subseteq\p{n+1}$ if and only if $2n-3-\delta\le n+1$, hence if and only if $1\le n-4\le \delta$. Taking into account of \cite[Theorem 1.1 and Example 2.5]{Ch--Ci} we know that the only admissible cases are either $m=6$ when $5\le n\le 6$ (in this case $Y_6$ is either the image of the Segre embedding of $\p2\times\p2$ inside $\p8$ or its general hyperplane section) or $m=8$ when $n=5$ (in this case $D_8$ is the image of the $2$--uple embedding of $\p3$ inside $\p9$).

In the first case the homogeneous ideals of the $D_6$'s are generated by quadrics and the syzygies among them are linear: in particular, they have property $N_2$ (see \cite{Kw--Pa} for details on such property).  Thus \cite[Theorem 1.2]{Kw--Pa} implies that the homogeneous ideals of the $Y_6$'s  inside $\p6$ and $\p7$ respectively are generated by quadric and cubic forms. Again we have smooth cubic hypersurfaces (in  $\p6$ and $\p7$) containing the $Y_6$'s which, consequently, support rank two orientable instanton bundles with quantum number $\quantum=1$: each such hypersurface has a Steiner pfaffian representation of size $8$.

Let us spend a few words also on the second case. In this case the general projection $Y_8\subseteq\p6$ is not contained in any cubic, but for a special choice of $H$ there is a $3$--dimensional space of cubics through it: see \cite{J--K--K}. Several computations made using {\em Macaulay2} (see \cite{Gr--St}) seems to suggest the impossibility of finding smooth cubics in such a space, but we have no proof for this fact.
\end{example}

Now we turn our attention to lower dimensional hypersurfaces, dealing with Fano threefolds in $\p 4$.

\begin{example}
\label{eFano}
Let $X\subseteq\p4$ be a smooth hypersurface of degree $d$. 

Let $2\le d\le 3$. In \cite[Theorems C and D]{Fa}, the author proves that $X$ supports rank two bundles $\cA$ with $c_1(\cA)=d-3$, $c_2(\cA)h=c\ge d-1$, $h^0\big(\cA\big)=h^1\big(\cA(-h)\big)=0$ for each integer $c\ge d-1$. 
Thus it is easy to check that $\cE:=\cA(h)$ is a rank two orientable $h$--instanton bundle on $X$ with quantum number $\quantum=c-d+1$. We deduce that $\cE$ fits into sequence \eqref{seqPfaffian} with $\cF$ as in sequence \eqref{seqOmega} where $\quantum=c-d+1$, $a=h^0\big(\cA(h)\big)$, $b=h^1\big(\cA(h)\big)$. In both cases, when $c=d-1$, then $\quantum=0$, hence $b=0$ because $\cE$ turns out to be Ulrich. When $c\ge d=2$ we can even find $\cE$ such that if $c=2$, then $a=2$ and $b=0$, and if $c\ge3$, then $a=1$ and $b=2c-5$.

Assume finally that $d=4$ and that $X$ is ordinary, i.e. there is a line $L\subseteq X\subseteq\p4$ such that $\cN_{L\vert X}\cong\cO_{\p1}\oplus\cO_{\p1}(-1)$. Thanks to \cite[Theorem 1.8]{An--Ca1} we still know the existence of a rank two orientable $h$--instanton bundles $\cE$ on $X$. 

In particular, in the aforementioned cases, $X$ has a Steiner pfaffian representation of size $2(d+\quantum)$ for each integer $\quantum\ge0$.
\end{example}

\appendix
\section{Existence of $h$--instanton sheaves using {{Macaulay2}}}
\label{sMacaulay}

As pointed out in Example \ref{eCubic4}, in this section we write a Macaulay code (see \cite{Gr--St}) which allows us to show that there are smooth cubic threefolds $X\subseteq\p5$ containing the isomorphic image $D_8$ of the del Pezzo surface $Y_8\subseteq\p8$.
\medskip

\beginOutput
i1 : R=QQ[x_0,x_1,x_2,x_3,x_4,x_5,x_6,x_7,x_8,x_9];
\endOutput
\beginOutput
i2 : R1=R/(ideal(x_9));\\
\endOutput
\beginOutput
i3 : P5=QQ[z_0,z_1,z_2,z_3,z_4,z_5];\\
\endOutput

\medskip
In order to check that $D_8$ is contained in a smooth cubic fourfold, we start by constructing the ideals of $Y_8$ and $\sigma_2(Y_8)$ (see \cite[Section 4]{C--G--G} for more details). Consider the matrix

\beginOutput
i4 : A1=matrix\{\{x_0,x_1,x_2,x_3,x_4\},\{x_1,x_3,x_4,x_6,x_7\},$\cdot\cdot\cdot$\}\\
\              3        5\\
o4 : Matrix R1  <--- R1\\
\endOutput
\beginOutput
i5 : I8= minors(2,A1);\\
\emptyLine
o5 : Ideal of R1\\
\endOutput
\beginOutput
i6 : K=minors(3,A1);\\
\emptyLine
o6 : Ideal of R1\\
\endOutput

\medskip
Now we construct our center of projection $\textrm{CP}$ by choosing six random linear forms in the variables $x_i$ and we check that $\textrm{CP}$ is disjoint from $\sigma_2(Y_8)$
\medskip

\beginOutput
i7 : for i from 0 to 5 do (f_i=random(1,R));\\
\endOutput
\beginOutput
i8 : CP=ideal(f_0,f_1,f_2,f_3,f_4,f_5);\\
\emptyLine
o8 : Ideal of R1\\
\endOutput
\beginOutput
i9 : saturate (CP+K)\\
o9 = ideal 1\\
\emptyLine
o9 : Ideal of R1\\
\endOutput

\medskip
Finally we project $D_8$ to $\p 5$ using the following method. We consider the quotient ring $R1/I(D_8)$ and we see the ideal of the projection as the kernel of the map
\[
\textrm{P5} \to R1/I(D_8)
\]
where $\textrm{P5}$ represent the coordinate ring of $\p 5$.
\medskip

\beginOutput
i10 : Rbar = R1/I8;\\
\endOutput
\beginOutput
i11 : CP = substitute(CP, Rbar);\\
\emptyLine
o11 : Ideal of Rbar\\
\endOutput
\beginOutput
i12 : J=kernel map (Rbar, P5, gens CP);\\
\emptyLine
o12 : Ideal of P5\\
\endOutput
\beginOutput
i13 : M=gens J;\\
\               1        19\\
o13 : Matrix P5  <--- P5\\
\endOutput

\medskip
The ideal $\textrm{J}$ is generated by $7$ cubics and $12$ quartics. The first generator represents a cubic $X\subseteq\p 5$ that contains $D_8$. Finally we check that the Hilbert polynomial of the singular locus of $X$ vanishes. 
\medskip

\beginOutput
i14 : C=ideal(M_\{0\});\\
\emptyLine
o14 : Ideal of P5\\
\endOutput
\beginOutput
i15 : hilbertPolynomial(ideal (singularLocus(C)))\\
o15 = 0\\
o15 : ProjectiveHilbertPolynomial\\
\endOutput

\medskip
The code can be easily adapted to deal with the del Pezzo surface $D_7$. The ideal of $D_7$ is generated by the $2\times 2$ minors of the matrix $A_2$ which is obtained by removing the fourth column from $A_1$. 

Using the same projection method one obtains that the ideal of $Y_7$ is generated by $13$ cubics and one quartic. It is possible to check that the ideal of the projection contains a cubic form corresponding to a fourfold $X$ through $Y_7$ and whose singular locus has vanishing Hilbert polynomial, hence such an $X$ is smooth.

\bigskip
\noindent
Vincenzo Antonelli,\\
Dipartimento di Scienze Matematiche, Politecnico di Torino,\\
c.so Duca degli Abruzzi 24,\\
10129 Torino, Italy\\
e-mail: {\tt vincenzo.antonelli@polito.it}

\bigskip
\noindent
Gianfranco Casnati,\\
Dipartimento di Scienze Matematiche, Politecnico di Torino,\\
c.so Duca degli Abruzzi 24,\\
10129 Torino, Italy\\
e-mail: {\tt gianfranco.casnati@polito.it}

\end{document}